\documentclass[UTF-8,reqno]{amsart}
\usepackage{enumerate}
\usepackage{mhequ}
\usepackage{amssymb,url,color, booktabs,nccmath}
\usepackage[left=3.2cm,right=3.2cm,top=4cm,bottom=4cm]{geometry}
\usepackage{mathrsfs}
\usepackage{enumitem,dsfont}

\usepackage{color}
\usepackage[colorlinks=true]{hyperref}
\hypersetup{
    linkcolor=blue,          
    citecolor=red,        
    filecolor=blue,      
    urlcolor=cyan
}

\definecolor{darkergreen}{rgb}{0.0, 0.5, 0.0}

\numberwithin{equation}{section}

\newcommand{\be}{\begin{eqnarray}}
\newcommand{\ee}{\end{eqnarray}}
\newcommand{\ce}{\begin{eqnarray*}}
\newcommand{\de}{\end{eqnarray*}}
\newtheorem{theorem}{Theorem}[section]
\newtheorem*{theorem*}{Theorem}
\newtheorem{lemma}[theorem]{Lemma}
\newtheorem{proposition}[theorem]{Proposition}
\newtheorem{Examples}[theorem]{Example}
\newtheorem{corollary}[theorem]{Corollary}

\newtheorem{definition}[theorem]{Definition}
\theoremstyle{definition}
\newtheorem{remark}[theorem]{Remark}

\newcommand{\cdummy}{\cdot}

\newcommand{\tmmathbf}[1]{\ensuremath{\mathbf{#1}}}
\newcommand{\tmop}[1]{\ensuremath{\operatorname{#1}}}

\DeclareMathOperator{\supp}{supp}

\def\eps{\varepsilon}

\def\<{{\langle}}
\def\>{{\rangle}}
\def\({{\Big(}}
\def\){{\Big)}}

\def\bx{{\mathbf{x}}}

\def\dif{{\mathord{{\rm d}}}}

\def\no{\nonumber}
\def\={&\!\!=\!\!&}

\def\cL{{\mathcal L}}

\def\cP{{\mathcal P}}

\def\cT{{\mathcal T}}

\def\mN{{\mathbb N}}

\def\mR{{\mathbb R}}

\def\1{{\mathbf{1}}}

\def\E{\mathbf E}

\def\geq{\geqslant}
\def\leq{\leqslant}

\def\div{\mathord{{\rm div}}}

\def\eps{\varepsilon}

\def\<{{\langle}}
\def\>{{\rangle}}
\def\({{\Big(}}
\def\){{\Big)}}

\def\bx{{\mathbf{x}}}

\def\dif{{\mathord{{\rm d}}}}

\def\no{\nonumber}
\def\={&\!\!=\!\!&}
\def\bt{\begin{theorem}}
\def\et{\end{theorem}}
\def\bl{\begin{lemma}}
\def\el{\end{lemma}}
\def\br{\begin{remark}}
\def\er{\end{remark}}
\def\bx{\begin{Examples}}
\def\ex{\end{Examples}}
\def\bd{\begin{definition}}
\def\ed{\end{definition}}
\def\bp{\begin{proposition}}
\def\ep{\end{proposition}}
\def\bc{\begin{corollary}}
\def\ec{\end{corollary}}

\def\geq{\geqslant}
\def\leq{\leqslant}

\def\div{\mathord{{\rm div}}}

 \def\R{\mathbb R}
 \def\R{\mathbb R}    
\def\N{\mathbb N}  
   
\def\<{\langle} \def\>{\rangle}

\allowdisplaybreaks

\begin{document}

\title[]{A class of supercritical/critical singular stochastic PDE\lowercase{s}: existence, non-uniqueness, non-Gaussianity,\\non-unique ergodicity}

\author{Martina Hofmanov\'a}
\address[M. Hofmanov\'a]{Fakult\"at f\"ur Mathematik, Universit\"at Bielefeld, D-33501 Bielefeld, Germany}
\email{hofmanova@math.uni-bielefeld.de}

\author{Rongchan Zhu}
\address[R. Zhu]{Department of Mathematics, Beijing Institute of Technology, Beijing 100081, China; Fakult\"at f\"ur Mathematik, Universit\"at Bielefeld, D-33501 Bielefeld, Germany}
\email{zhurongchan@126.com}

\author{Xiangchan Zhu}
\address[X. Zhu]{ Academy of Mathematics and Systems Science,
Chinese Academy of Sciences, Beijing 100190, China; Fakult\"at f\"ur Mathematik, Universit\"at Bielefeld, D-33501 Bielefeld, Germany}
\email{zhuxiangchan@126.com}
\thanks{
This project has received funding from the European Research Council (ERC) under the European Union's Horizon 2020 research and innovation programme (grant agreement No. 949981).  The financial support by the DFG through the CRC 1283 ``Taming uncertainty and profiting
 from randomness and low regularity in analysis, stochastics and their applications'' is greatly acknowledged. R.Z. and X.Z. are grateful to
 the financial supports   by National Key R\&D Program of China (No. 2022YFA1006300). 
 R.Z. gratefully acknowledges financial support from the NSFC (No.
  12271030).  X.Z. is grateful to
 the financial supports  by National Key R\&D Program  of China (No. 2020YFA0712700) and the NSFC (No. 12288201, 12090014) and
  the support by key Lab of Random Complex Structures and Data Science,
 Youth Innovation Promotion Association (2020003), Chinese Academy of Science.
}

\begin{abstract}
We study the surface quasi-geostrophic equation with an irregular spatial perturbation
$$
\partial_{t }\theta+ u\cdot\nabla\theta = -\nu(-\Delta)^{\gamma/2}\theta+ \zeta,\qquad u=\nabla^{\perp}(-\Delta)^{-1}\theta,
$$
on $[0,\infty)\times\mathbb{T}^{2}$, with  $\nu\geq 0$, $\gamma\in [0,3/2)$ and $\zeta\in B^{-2+\kappa}_{\infty,\infty}(\mathbb{T}^{2})$ for some $\kappa>0$. This covers the case of $\zeta= (-\Delta)^{\alpha/2}\xi$ for $\alpha<1$ and $\xi$  a spatial white noise on $\mathbb{T}^{2}$. Depending on the relation between $\gamma$ and $\alpha$, our setting is   subcritical, critical or supercritical in the language of Hairer's regularity structures \cite{Hai14}. Based on purely analytical tools from convex integration and without the need of any probabilistic arguments including renormalization, we prove existence of infinitely many analytically weak solutions in $L^{p}_{\rm{loc}}(0,\infty;B_{\infty,1}^{-1/2})\cap C_{b}([0,\infty);B^{-1/2-\delta}_{\infty,1})\cap C^{1}_{b}([0,\infty);B^{-3/2-\delta}_{\infty,1})$ for all  $p\in [1,\infty)$ and $\delta>0$. We are able to prescribe an initial as well as a terminal condition at a finite time $T>0$, and to construct steady state, i.e. time independent, solutions. In all cases, the solutions are non-Gaussian, but we may as well prescribe Gaussianity at some given times. Moreover, a coming down from infinity with respect to the perturbation and the initial condition holds. Finally, we show that the our solutions generate statistically stationary solutions as limits of ergodic averages, and we obtain existence of infinitely many non-Gaussian time dependent ergodic stationary solutions. We also extend our results to a more general class of singular SPDEs.
 \end{abstract}

\date{\today}

\maketitle

\tableofcontents

\section{Introduction}

We are concerned with the question of well/ill-posedness as well as long time behavior of the  surface quasi-geostrophic (SQG) equation  on $\mathbb{T}^2$ driven by an irregular spatial  perturbation
\begin{equation}\label{eq2}
\begin{aligned}
\partial_{t }\theta+ u\cdot\nabla\theta &= -\nu\Lambda^{\gamma}\theta+ \zeta,\\
u=\nabla^\perp\Lambda^{-1}\theta&=(-\partial_2\Lambda^{-1}\theta,\partial_1\Lambda^{-1}\theta)=(-\mathcal{R}_2\theta, \mathcal{R}_1\theta),\\
\theta(0)&=\theta_{0},
\end{aligned}
\end{equation}
where $\Lambda=(-\Delta)^{1/2}$, $\nu\geq0$, $\gamma\in[0,3/2)$ and   $\mathcal{R}=(\mathcal{R}_1,\mathcal{R}_2)$ is the pair of Riesz transforms. Here the unknown scalar function $\theta$ denotes the potential temperature in the geophysical fluid dynamics.
The spatial perturbation $\zeta$ is merely assumed to belong to $B^{-2+\kappa}_{\infty,\infty}(\mathbb{T}^{2})$ for some $\kappa>0$. This particularly includes the case of $\zeta=\Lambda^{\alpha}\xi$, where $\xi$ is a space white noise and  $\alpha<1$. While this is  the main example we have in mind, our construction is entirely deterministic and does not use any probability theory, treating therefore an arbitrary distribution $\zeta$ in this regularity class.

In the language of regularity structures \cite{Hai14}, our setting includes the supercritical, critical as well as subcritical regime, meaning, the smoothing of the linear part does not need to dominate the nonlinear term. This can be determined by a simple scaling argument. In the case of \eqref{eq2} with $\zeta=\Lambda^{\alpha}\xi$,  it can be seen that $\tilde\xi(x):=\lambda\xi(\lambda x)$ is again a space white noise in two dimensions. Hence letting $\tilde\theta(t,x):=\lambda^{1+\alpha-\gamma}\theta(\lambda^{\gamma}t,\lambda x)$ and $\tilde u(t,x):=\lambda^{1+\alpha-\gamma}u(\lambda^{\gamma}t,\lambda x)$ leads to
$$
\partial_{t}\tilde\theta +\lambda^{2\gamma-2-\alpha}\tilde u\cdot\nabla\tilde\theta = - \nu\Lambda^{\gamma}\tilde\theta +\Lambda^{\alpha}\tilde\xi.
$$
Consequently, the equation is subcritical if $2\gamma-2-\alpha>0$, critical if $2\gamma-2-\alpha=0$ and supercritical if $2\gamma-2-\alpha<0$. As our results are valid for any $\gamma\in [0,3/2)$, $\alpha<1$ and $\nu\geq 0$, they include all these three regimes.

This notion of supercriticality/criticality/subcriticality can be also understood in terms of regularity as follows. In view of  the regularity of space white noise it holds $\zeta=\Lambda^{\alpha}\xi\in B^{-1-\alpha-\kappa}_{\infty,\infty}$ for any $\kappa>0$. Hence, it follows formally from Schauder estimates that the solution $\theta$ belongs to $B^{\gamma-1-\alpha-\kappa}_{\infty,\infty}$ and $u$ has the same regularity. Hence, if this regularity is negative, i.e. $\theta$ and $u$ are not function valued, we obtain again formally that the nonlinear term  belongs to $B^{2\gamma-3-2\alpha-\kappa}_{\infty,\infty}$ provided it can be made sense of. If this regularity is smaller/equal/bigger than the regularity of the Gaussian noise, we are in the supercritical/critical/subcritical case.  Thus, intuitively,  the subcritical case can be understood as a perturbation of the solution to the linear equation, which is Gaussian.

\subsection{Supercritical/critical singular SPDEs}

Hairer's regularity structures theory \cite{Hai14} and the paracontrolled distributions method by Gubinelli, Imkeller and Perkowski  \cite{GIP15}
	 made it possible to study a large class of singular stochastic PDEs in the subcritical regime. The key ideas of these theories are to view nonlinearities  as perturbations of Gaussian terms and to use the structure of the solution to give a meaning to the terms which are not classically well-defined. These terms are well-defined with the help of renormalization for the enhanced noise, i.e. the noise and the higher order terms appearing in the decomposition of the equations. However, these tricks break down in the critical and supercritical settings as the nonlinear terms can no longer be viewed as perturbations of Gaussian random variables. Formally, there are infinitely many terms that require renormalization.
	
	 For this reason, the rigorous understanding of critical and supercritical models is very limited. In the critical regime, the 2D isotropic Kardar--Parisi--Zhang (KPZ) equation was considered in \cite{CD20, CSZ20, Gu20} via the Cole--Hopf transform. The 2D anisotropic KPZ equation was studied in \cite{CES21, CET21} based on the related invariant measure given by a Gaussian measure. Recently, for a two dimensional  critical Navier--Stokes equations, tightness of approximate stationary solutions and non-triviality of the limit has been established in a weak coupling regime, i.e. with a vanishing constant in front of the nonlinearity, in  \cite{CK21}. We also mention the recent progress in \cite{CSZ21} for the critical 2D stochastic heat flow. In the supercritical regime, a recent series of works \cite{MU18, CCM20, DGRZ20, CNN20}   studied the KPZ equation for $d\geq 3$. More precisely, the $d$-dimensional KPZ and the two dimensional critical Navier--Stokes equations were understood as the limit of  regularized equations and in the approximate equations a vanishing scaling constant was included in front of the nonlinearity.
Additionally, the limit is then given by a stochastic heat equation whose law is  Gaussian  (except \cite{CSZ21}).

In \cite{GJ13},  global stationary probabilistically weak solutions for stochastic fractional burgers equation were constructed both in critical and supercritical regime via the energy solution method, which depends on the related invariant measure given by the law of space white noise.
Apart from the models mentioned above, another classical example of a singular SPDE is
	the dynamical $\Phi^4_d$ model. The
	critical case corresponding to $d=4$ is  out of reach at this moment. In a remarkable recent
	work \cite{ADC21}, it was shown that at large scales the formal $\Phi^4_4$ invariant measure
	 is Gaussian.

To summarize, the available results for singular SPDEs in the critical and supercritical regime rely on a particular structure of the models using a certain transform, a vanishing scaling constant in front of the nonlinearity (the so-called weak coupling regime)  or  probabilistic arguments using a Gaussian invariant measure. Our technique only requires the regularity of the driving noise and can be applied to the equation driven by deterministic forcing or a general noise, not necessarily Gaussian and without a Gaussian invariant measure. Moreover, we do not need to include a small scaling constant in front of  the nonlinearity and the constructed solutions are not trivial, i.e. not Gaussian.

We also mention the stochastic quantization of the 4D Yang--Mills theory, which plays an important role in
the standard model of quantum mechanics. It is a critical model and the construction of 4D Yang--Mills field is related to one of the Millennium Prize Problems (c.f. \cite{JW06}). Note that local unique solutions to the Langevin dynamics for the Yang--Mills model
	on the two and three dimensional continuous torus were recently constructed in \cite{CCHS20,
	CCHS22} by using the theory of regularity structures.

\subsection{Convex integration}

Since the seminal works by De~Lellis and Sz\'ekelyhidi \cite{DelSze2,DelSze3,DelSze13}, convex integration has become a powerful tool in the context of fluid dynamics. It led to a large number of striking results concerning e.g. the Euler equations \cite{BDLIS16,BDSV,Ise}, the Navier--Stokes equations \cite{BCV18,BV19a,CL20,CL21}, the compressible Euler equations \cite{C14, F16}, power law fluids \cite{BMS20} and others. The method was also successfully applied in the context of the deterministic SQG equation in \cite{BSV19, CKL21, IM21}. Various results were  obtained  in the stochastic setting in \cite{BFH20,CFF19,HZZ19, HZZ20,HZZ21a,HZZ21,LZ22, RS21,Ya20a,Ya20b,Ya21,Ya21b,Ya21c}. Let us particularly point out our previous work \cite{HZZ21} where convex integration was  for the first time applied in the setting of singular SPDEs. In particular, we proved global existence and non-uniqueness  to the 3D Navier--Stokes system perturbed by a space-time white noise. This is a subcritical problem. As the noise was also rough in time, it was necessary to use the power of renormalization combined with paracontrolled calculus in order to make sense of the equations.

The method of convex integration brings a completely new perspective in the field of  (S)PDEs. If we try to solve \eqref{eq2} by classical PDE arguments we see the problem: formally by Schauder estimates the regularity of a solution to \eqref{eq2} is  $B_{\infty,\infty}^{\gamma-1-\alpha-\kappa}$. Hence, only if $\gamma>1+\alpha+\kappa$  the nonlinear term is analytically well-defined as in this case $\theta$ and $u$ are function valued. The usual fixed point argument therefore breaks down for  $\gamma\leq 1+\alpha$. For $1+\alpha/2<\gamma\leq 1+\alpha$, i.e. in the subcritical case, the regularity structures theory gives a local existence and uniqueness of solutions. However, it does not work in the critical and supercritical regime.

As we discuss below, the nonlinear term is  well-defined even for $\theta\in \dot{H}^{-1/2}$ which leads to a notion of weak solution. In the case of $\gamma>1+2\alpha$ the corresponding estimate for the $H^{-1/2}$-norm of the solution $\theta$ can be closed leading to the existence of global-in-time weak solutions.   Nevertheless, as seen above, this is way above the subcriticality level and in this setting the solution is actually function valued by Schauder estimates anyway, i.e. no renormalization is required. Hence, this is of no help to the fixed point argument typically used in proofs of local well-posedness for singular SPDEs in the subcritical regime. Also critical and supercritical setting is out of reach by this method.

Let us explain the main ideas of the convex integration approach to the SQG equation. It is a more constructive way towards existence of solutions  than the usual PDE arguments. It is an iterative procedure which relies on building blocks tailor made to the PDE at hand. Different PDEs require different structure of these blocks. The building blocks for the SQG equation in \cite{CKL21} are frequency localized oscillatory waves  $\cos(\lambda_{n+1} l\cdot x)$ for suitable $l\in \mathbb{R}^{2}$. The  iteration starts from an approximate solution, i.e. a solution satisfying the equation up to a certain error. At each iteration step  $n+1$, one adds a perturbation of the form $\tilde f_{n+1}=\sum_{l}a_{l,n+1}(x)\cos(\lambda_{n+1} l\cdot x)$ for well-chosen amplitude functions $a_{l,n+1}(x)$ supported on much smaller frequencies {than the oscillatory wave $\cos(\lambda_{n+1} l\cdot x)$. Consequently, the frequencies of $\tilde{f}_{n+1}$ are supported in an annulus of size $\lambda_{n+1}$. Hence its product with the previous iteration, which is supported on much smaller frequencies of size at most $\lambda_{n}$, is always well defined. The idea is to define the amplitude functions so that  after the perturbation is plugged into the nonlinearity, the corresponding   product reduces the error. This way, the building blocks propagate oscillations through the nonlinearity and no ill-definiteness problem for the product of the perturbation terms appears due to the special structure of the building blocks.} Repeating this procedure countably many times, one obtains a solution to the desired equation in the limit.

In our proof, we add the noise  scale by scale during the iteration. This way, we are able to preserve the frequency localization of the perturbations. Furthermore and  rather surprisingly, even a noise of such a low regularity  as in the critical and supercritical regime can be viewed as a small error in this procedure. As a matter of fact, since the dissipative term is anyway treated as a lower order term and not used to gain regularity, we may drop it altogether.
Moreover, we may also consider a time dependent noise provided it is function valued with respect to time. A time irregular noise requires further ideas and will be considered in a future work. In addition, it turns out that our approach can also be applied to other models, see Section \ref{other} for more details.  As we mentioned above, the stochastic quantization of the 4D Yang--Mills theory is a critical SPDE. We hope that our work and the convex integration method can shed some light on the understanding of this important model.

\subsection{Main results}

As in the related deterministic works \cite{BSV19, CKL21, IM21}, we make use of the fact that the nonlinearity in \eqref{eq2} is well-defined for $\theta\in L^{2}_{\rm{loc}}(0,\infty;\dot{H}^{-1/2})$. More precisely, for any $\psi\in C^{\infty}(\mathbb{T}^{2})$  it holds
$$
\langle u\cdot\nabla\theta,\psi\rangle =\frac12 \langle \theta,[\mathcal{R}^{\perp}\cdot,\nabla\psi]\theta\rangle,
$$
and the commutator $[\mathcal{R}^{\perp}\cdot,\nabla\psi]=-[\mathcal{R}_2\cdot,\partial_1\psi]+[\mathcal{R}_1\cdot,\partial_2\psi]$ maps $\dot{H}^{-1/2}$ to $\dot{H}^{1/2}$ (c.f. \cite[Proposition 5.1]{CKL21}). Hence,
we aim at solving \eqref{eq2} in the following sense.

\begin{definition}\label{d:1}
  We say that $\theta \in L^{2}_{\rm{loc}}(0,\infty;\dot{H}^{-1/2})\cap C([0,\infty);B^{-1}_{\infty,1})$ is a weak solution to \eqref{eq2}
  provided
for any $t\geq0$
\begin{align*}
\langle \theta(t),\psi\rangle+\int_0^t\frac12\langle \Lambda^{-1/2}\theta,\Lambda^{1/2}[\mathcal{R}^\perp\cdot,\nabla \psi] \theta\rangle\dif t
 = \langle \theta_0,\psi\rangle + \int_{0}^{t}\langle\zeta -\nu\Lambda^{\gamma } \theta
,\psi\rangle\dif t
\end{align*}
holds  for every $\psi \in C^{\infty}(\mathbb{T}^{2})$.
  In case of $\zeta=\Lambda^{\alpha}\xi$ for  a spatial white noise  $\xi$ on a probability space $(\Omega,\mathcal{F},\mathbf{P})$, we require in addition that $\theta$ is $\mathcal{F}$-measurable.
\end{definition}

In the first step, we are interested in the question of well/ill-posedness. Our  results may be briefly summarized as follows and we refer to Theorem~\ref{thm:6.1}, Theorem~\ref{thm:6.11}, Corollary~\ref{cor:law} and Corollary~\ref{c:coming1} for precise formulations and further details.

\begin{theorem}
There exist infinitely many
\begin{enumerate}
\item weak solutions on $[0,\infty)$ for any  prescribed initial condition $\theta_{0}\in C^{\eta}$ $\mathbf{P}$-a.s., $\eta>1/2$,
\item weak solutions on $[0,T]$ for any prescribed initial and terminal  condition $\theta_{0}, \theta_{T}\in C^{\eta}$ $\mathbf{P}$-a.s., $\eta>1/2$, $T\geq4$.
\end{enumerate}
Moreover, the solutions are non-Gaussian and satisfy a coming down from infinity with respect to the noise as well as the initial condition.
\end{theorem}

 The reason why we require a smooth initial condition compared to \cite{HZZ21} is that we are now in a more singular situation, i.e. critical or supercritical regime in the sense of Hairer's regularity structure while  \cite{HZZ21} only deals with the subcritical regime. In the present paper, we do not rely on a  mild formulation to give meaning to the equation, but instead we understand the equation in a weak sense. From Definition~\ref{d:1} we know that to give a meaning to the solution it only suffices that it stays in $\dot H^{-1/2}$. However, in the convex integration estimate we still need to bound the product of  $u$ and $\theta$. Formally the main irregular part is in $\dot H^{-1/2}$ and we need to estimate the initial part multiplied by the main irregular part, which requires the condition $\eta>1/2$.

The core of our proofs is an iterative convex integration procedure in the spirit of \cite{CKL21}. In the latter work,  existence and non-uniqueness of steady state weak solutions to the SQG equation \eqref{eq2} with $\zeta=0$ was proved. We show that the spatial irregular perturbation $\zeta$ can be essentially treated as a lower order term and added scale by scale as one proceeds through the iteration. In a similar fashion we can treat the  initial as well the terminal condition. The coming down from infinity  provides sharp bounds independent $\zeta$ and of the initial condition.

As the next step, we study the long time behavior, existence of stationary solutions and the associated ergodic structure. Due to the lack of uniqueness, we understand stationarity in the sense of shift invariance of the law on the space of trajectories, see also \cite{BFHM19, BFH20e,FFH21}. To be more precise, we say that a weak solution $\theta$ is a stationary solution provided the probability laws $\cL[\theta]$ and $\cL[\theta(\cdot+t)]$ of $\theta$ and $\theta(\cdot+t)$, respectively,
coincide as probability measures on $C({\mR};B^{-1}_{\infty,1})$ for all $t\geq 0$. Accordingly, we say that a stationary solution $\theta$ is ergodic, provided
$
\cL[\theta](B)=1$ or $\cL[\theta](B)=0$
for any $B\subset C({\mR};B^{-1}_{\infty,1})$ Borel and shift invariant.

Note that this setting is different from the usual setting of invariance with respect to a Markov semigroup. The latter notion can be  applied to problems with uniqueness, i.e. where the Markov property holds. The construction of invariant measures then additionally requires the  Feller property which corresponds to continuous dependence on initial condition. Since non-uniqueness holds true for \eqref{eq2}, and it is unclear whether Markov solutions can be obtained by a selection procedure, we employ the more general notion of invariance with respect to shifts on trajectories.
Another advantage is that continuity of the shift operators comes for free and therefore there is no need for any Feller property.
Also with this notion of invariance, the existence of an ergodic stationary solution implies the validity of the so-called ergodic hypothesis, i.e. the fact that ergodic averages along trajectories of the ergodic solution converge to the ensemble average given by its law.

The following is proved Theorem~\ref{th:s1}, Corollary~\ref{cor:4.3}, Theorem~\ref{thm:4.5},  Corollary~\ref{cor:4.6} and Theorem~\ref{thm:4.1}.

\begin{theorem}
There exist {infinitely many non-Gaussian}
\begin{enumerate}
\item   stationary solutions,
\item   ergodic stationary solutions,
\item  steady state, i.e. time independent, solutions.
\end{enumerate}
Moreover, the ergodic stationary solutions are time dependent. The point (3)  additionally implies existence and non-uniqueness of solutions to the corresponding elliptic and wave equation.
\end{theorem}

\subsection{Further relevent literature} Stochastic surface quasi-geostrophic equation driven by a trace-class
	noise, was studied in \cite{RZZ14, RZZ15, ZZ17}. Global well-posedness  and ergodicity were obtained in \cite{RZZ15} for $\gamma\in (1,2)$. It is
	shown in \cite{RZZ14} that the linear multiplicative noise prevents the system from exploding with a large probability. More recently, in \cite{FS21} local well-posedness was proved for the surface quasi-geostrophic equation driven by space-time white
	noise in the subcritical regime (i.e. $\gamma>4/3$) using the theory of regularity structures.
	
\section{Notations}
\label{s:not}

\subsection{Function spaces}

  Throughout the paper, we employ the notation $a\lesssim b$ if there exists a constant $c>0$ such that $a\leq cb$, and we write $a\simeq b$ if $a\lesssim b$ and $b\lesssim a$. $\mN_{0}:=\mN\cup \{0\}$. Given a Banach space $E$ with a norm  $\|\cdot\|_E$, we write  $CE$ or $C([0,\infty);E)$ to denote the space of continuous functions from $[0,\infty)$ to $E$.  We  define  $C^1E$ as the space of  continuous functions with locally bounded first order derivative from $[0,\infty)$ to $E$.  Similarly, for $T>0$ we use $C^1_TE$ as the space of  continuous functions with bounded first order derivative from $[0,T]$ to $E$. For $\alpha\in(0,1)$ we  define $C^\alpha_TE$ as the space of $\alpha$-H\"{o}lder continuous functions from $[0,T]$ to $E$, endowed with the norm $\|f\|_{C^\alpha_TE}=\sup_{s,t\in[0,T],s\neq t}\frac{\|f(s)-f(t)\|_E}{|t-s|^\alpha}+\sup_{t\in[0,T]}\|f(t)\|_{E}$. We also write $C_b([0,\infty);E)$ for functions in $C([0,\infty);E)$ such that $\|f\|_{C_{b}([0,\infty);E)}:=\sup_{t\in [0,\infty)}\|f(t)\|_E<\infty$. For $\beta\in (0,1]$ then define $C_b^\beta([0,\infty);E)$ as functions in $C^\beta([0,\infty);E)$ such that $\|f\|_{C_b^\beta([0,\infty);E)}:=\sup_{t\in [0,\infty)}\|f\|_{C^\beta_tE}<\infty$. Similarly we  define $C_{b}(\mR;E)$,  $C_b^{\beta}(\mR;E)$ with $[0,\infty)$ replaced by $\mR$.
  For $p\in [1,\infty]$ we write $L^p_TE=L^p(0,T;E)$ for the space of $L^p$-integrable functions from $[0,T]$ to $E$, equipped with the usual $L^p$-norm. We also use $L^p_{\mathrm{loc}}([0,\infty);E)$ to denote the space of functions $f$ from $[0,\infty)$ to $E$ satisfying $f|_{[0,T]}\in L^p_T E$ for all $ T>0$.
 We use $C_0^\infty$ to denote the set of  $C^\infty$ functions with mean zero.
For $f$ on $\mathbb{T}^2$ we follow the Fourier transform convention $\mathcal{F}f(k)=\hat{f}(k)=\frac1{(2\pi)^2}\int_{\mathbb{T}^2}f(x)e^{-ix\cdot k}\dif x.$ For $s\in\mathbb{R}$,  we set $\dot{H}^s=\{f:\|(-\Delta)^{s/2}f\|_{L^2}<\infty\}$ and equip it with the norm  $\|f\|_{\dot{H}^s}=(\sum_{0\neq k\in\mathbb{Z}^2}|k|^{2s}|\hat{f}(k)|^2)^{1/2}$.  We use $(\Delta_{i})_{i\geq -1}$ to denote the Littlewood--Paley blocks corresponding to a dyadic partition of unity.
Besov spaces on the torus with general indices $\alpha\in \R$, $p,q\in[1,\infty]$ are defined as
the completion of $C^\infty(\mathbb{T}^{d})$ with respect to the norm
$$
\|u\|_{B^\alpha_{p,q}}:=\left(\sum_{j\geq-1}2^{j\alpha q}\|\Delta_ju\|_{L^p}^q\right)^{1/q},$$
with the usual interpretation as the $\ell^\infty$-norm when $q=\infty.$
The H\"{o}lder--Besov space $C^\alpha$ is given by $C^\alpha=B^\alpha_{\infty,\infty}$. Let $s\in\mathbb{R}$. If $f\in B^s_{\infty,1}$ with mean zero, then $f\in \dot{H}^{s}\cap C^s$
and
\begin{align}\label{besov}\|f\|_{\dot{H}^{s}}+\|f\|_{C^{s}}\lesssim \|f\|_{B^s_{\infty,1}}.\end{align}

As in \cite{CKL21} we denote
    $$v \approx w\quad \textrm{if }\quad v=w+\nabla^\perp p$$
    holds for some smooth scalar function $p$.
Define the projection operator for $\lambda\geq1$
$$\widehat{(P_{\leq\lambda}g)}(k)=\psi \left(\frac{k}{\lambda}\right)\hat{g}(k).$$
Here $\psi\in C_c^\infty(\mathbb{R}^2)$ satisfies $\psi(k)=0$ for $|k|\geq1$ and $\psi(k)=1$ for $|k|\leq 1/2$. We also introduce the Riesz-type transform $\mathcal{R}_j^o$, $j=1,2$ as follows
$$\hat{\mathcal{R}}_1^o(k_1,k_2)=\frac{25(k_2^2-k_1^2)}{12|k|^2},\quad \hat{\mathcal{R}}_2^o(k_1,k_2)=\frac{7(k_2^2-k_1^2)}{12|k|^2}+\frac{4k_1k_2}{|k|^2}.$$
We define the  norm
$$\|q\|_X:=\|q\|_{L^\infty}+\sum_{j=1}^2\|\mathcal{R}_j^oq\|_{L^\infty}.$$

\subsection{Probabilistic elements}\label{s:2.2}
Let $\xi$ be a white noise on $\mathbb{T}^2$ with mean zero defined on a probability space $(\Omega,\mathcal{F},\mathbf{P})$. More precisely, $(\hat{\xi}(k))_{k\in\mathbb{Z}^2}$ is a complex-valued centered Gaussian process with covariance
$$\E[\hat{\xi}(k)\hat{\xi}(k')]=\frac1{(2\pi)^2}1_{k=-k'}$$
and such that $\overline{\hat{\xi}(k)}=\hat{\xi}(-k)$ for all $k,k'\in \mathbb{Z}^2$. This yields, using the Gaussian hypercontractivity together with the Besov embedding, that $\E[\|\xi\|_{C^{-1-\kappa}}^p]<\infty$ for all $\kappa>0$, $p\in[1,\infty)$.

\section{Solutions to the initial value problem}
\label{s:1.1}

We are concerned with the SQG equation driven by a fractional derivative of spatial white noise, namely
\begin{align}\label{eq1}
\begin{aligned}
\partial_t \theta + u \cdummy \nabla \theta &= - \Lambda^{\gamma} \theta +
  \Lambda^{\alpha}\xi,\\
  u=\nabla^\perp\Lambda^{-1}\theta&=(-\partial_2\Lambda^{-1}\theta,\partial_1\Lambda^{-1}\theta)=(-\mathcal{R}_2\theta, \mathcal{R}_1\theta),\\
  \theta(0)&=\theta_0,
  \end{aligned}
  \end{align}
where $\gamma\in[0,3/2)$, $\alpha\in [0,1)$ and $\xi$ is a space white noise. For notational simplicity we consider unitary viscosity $\nu=1$ but our results hold for all $\nu\geq0$. The goal of this section is to prove existence of infinitely many solutions to the initial value problem and also existence of infinitely many solutions with prescribed initial as well as terminal condition. This has  consequences regarding the probability laws of solutions at various  times and the so-called coming down from infinity discussed in Section~\ref{s:com}.

We proceed similarly as  in \cite{CKL21}, where existence of infinitely many steady state solutions to the SQG equation with $\xi=0$ was established.
To this end, we let $f = \Lambda^{- 1} \theta$ and observe that the equation solved by $f$ reads as
\begin{equation}\label{eq:f}
\begin{aligned}
 \nabla \cdummy (-\partial_t \mathcal{R}f + \Lambda f \nabla^{\perp} f)& =
   \nabla \cdummy ( \Lambda^{\gamma - 1} \nabla f) + \nabla \cdummy
   (-\mathcal{R} \Lambda^{- 1 + \alpha} \xi),\\
    f(0)&=\Lambda^{-1}\theta_0.
    \end{aligned}
    \end{equation}
   Here $\theta_0\in C^\eta$ for some $\eta>1/2$ with zero mean.
We first assume that there exists deterministic constants $L,N\geq 1$ such that
\begin{equation}\label{eq:u0}
\|\theta_{0}\|_{C^\eta}\leq N,\quad \|\xi\|_{C^{-1-\kappa}}\leq L.
\end{equation}
We keep this additional assumption on the initial condition and the noise throughout the convex integration step in Proposition~\ref{p:iteration} and relax it later in the proof of Theorem~\ref{thm:6.1}.

 Define $M_L=C(N^2+L)$ for a universal constant $C>0$ given below.
    We consider an increasing sequence $\{\lambda_n\}_{n\in\mathbb{N}_{0}}$ which diverges to $\infty$, and a sequence $\{r_n\}_{n\in \mathbb{N}}$  which is decreasing to $0$. We choose $a\in\mathbb{N},$ $ b\in\mathbb{N},$ $  \beta\in (0,1)$ and let
$$\lambda_n=a^{(b^n)}, \quad r_n=M_L\lambda_0^\beta\lambda_n^{-\beta},\quad \ell_n=\lambda^{-1}_n,\quad \mu_n=(\lambda_n\lambda_{n-1})^{1/2},$$
where  $\beta$ will be chosen sufficiently small and $a$  will be chosen sufficiently large.
We first assume that
$$\sum_{m\geq0}r^{1/2}_{m}\leq M^{1/2}_L+M^{1/2}_L\sum_{m\geq1}a^{\beta/2-mb\beta/2}\leq M_L^{1/2}+\frac{M_L^{1/2}}{1-a^{-b\beta/2}}\leq 3M^{1/2}_L,$$
and
$$
\sum_{m\geq1}\lambda_0^{\beta/2}\lambda^{-{\beta}/{4b}}_{m}\leq \lambda_0^{\beta/4}+\lambda_0^{\beta/4}\sum_{m\geq1}a^{\beta/4-mb\beta/4}\leq\lambda_0^{\beta/4}+\frac{\lambda_0^{\beta/4}}{1-a^{-b\beta/4}}\leq 3\lambda_0^{\beta/4},$$
which boils down to
\begin{align}\label{para}a^{b\beta}\geq 16.
\end{align}
More details on the choice of these parameters will be given below in the course of the construction.

The iteration is indexed by a parameter $n\in\mathbb{N}_0$. At each step $n$, a pair $(f_{\leq n},q_n)\in C^\infty_0\times C^\infty_0$ is constructed solving the following system
\begin{align}\label{induction ps}
\begin{aligned}
-\partial_t \mathcal{R}f_{\leqslant n} + \Lambda f_{\leqslant n}
   \nabla^{\perp} f_{\leqslant n} +\mathcal{R} \Lambda^{- 1 + \alpha}
   P_{\leq \lambda_n}\xi &\approx \Lambda^{\gamma - 1} \nabla f_{\leqslant n} +
   \nabla q_n,\\
   f_{\leqslant n}(0)&=P_{\leq \lambda_n/3}\Lambda^{-1}\theta_0.
   \end{aligned}
   \end{align}
We recall that the notation $\approx$ and $P_{\leq \lambda}$ were introduced in Section~\ref{s:not}.
{Taking divergence of} \eqref{induction ps} yields  an approximation of \eqref{eq:f} and  $q_{n}$ represents an error which shall eventually vanish with $n\to\infty$. {Due to the application of divergence, we only need to consider the mean-free parts of all the terms in \eqref{induction ps}. For notational simplicity, we do so without further modifying the notation.}

We observe that the noise as well as the initial condition are added in \eqref{induction ps} scale by scale. This permits to preserve the frequency localization of $f_{\leq n}$, $q_{n}$ as seen in the following iterative proposition.
Its proof is presented in Section~\ref{s:it}.

\begin{proposition}\label{p:iteration}
Let $L,N\geq 1$ and assume \eqref{eq:u0}. There exists a choice of parameters $a, b, \beta$ such that the following holds true: Let $(f_{\leq n},q_{n})$ for some $n\in\N_{0}$ be an $\mathcal{F}$-measurable solution to \eqref{induction ps} such that for any $t\geq 0$ the frequencies of $f_{\leq n}(t),q_n(t)$ are localized in a ball of radius $\leq 6\lambda_n$ and $\leq 12\lambda_n$, respectively, and for any $t\geq 0$
\begin{equation}\label{inductionv ps}
\|f_{\leq n}(t)\|_{B^{1/2}_{\infty,1}}\leq10M_0M_L^{1/2}
\end{equation}
for a constant $M_0$ independent of $a,b,\beta,M_L$, and
\begin{align}\label{eq:C1}
\|f_{\leq n}\|_{C^1_bB^{1/2}_{\infty,1}}\leq M_0M^{1/2}_L\lambda_n,
\end{align}
and for $\delta>\beta/2$
\begin{align}\label{induction wf}\|f_{\leq n}\|_{C_b^1B^{-1/2-\delta}_{\infty,1}}\leq M_L^{1/2}+M_0M^{1/2}_L\sum_{k=1}^n(\lambda_{k}^{-\delta}+\lambda_{k-1}^{{-3/2-\eta}})\leq M_L^{1/2}+6M_0M_L^{1/2},\end{align}
and
\begin{align}\label{eq:R}
\|q_n(t)\|_{X}\leq
r_{n},\quad t\in [2^{-n+1},\infty),
\end{align}
\begin{align}\label{bd:R}
\|q_n(t)\|_{X}\leq M_L+\sum_{k=1}^nr_{k}\leq 3M_L,\quad t\in[0, 2^{-n+1}).
\end{align}
Here we defined $\sum_{1\leq r\leq0}:=0.$
 Then    there exists an $\mathcal{F}$-measurable solution $(f_{\leq n+1},q_{n+1})$ which solves \eqref{induction ps} with $n$ replaced by $n+1$ and such that for any $t\geq 0$ the frequencies of $f_{\leq n+1}(t),q_{n+1}(t)$ are localized in a ball of radius $\leq 6\lambda_{n+1}$ and $\leq 12\lambda_{n+1}$, respectively, and
 \begin{equation}\label{iteration ps}
\|f_{\leq n+1}(t)-f_{\leq n}(t)\|_{B^{1/2}_{\infty,1}}\leq \begin{cases}
M_0 r_{n}^{1/2}+r^{1/2}_{n+1},& t\in [2^{-n+2},\infty),\\
M_0 M_L^{1/2}+r^{1/2}_{n+1},& t\in (2^{-n-1},2^{-n+2}),\\
r^{1/2}_{n+1},
&t\in [0,2^{-n-1}],
\end{cases}
\end{equation}
and for $\vartheta<(1/2+\beta/4b)\wedge \eta$, $t\in (2^{-n+2},\infty)$
\begin{align}\label{induction fn}\|f_{\leq n+1}(t)-f_{\leq n}(t)\|_{C^{\vartheta}}\leq M_0M_L^{1/2}\lambda_0^{\beta/2}\lambda^{-{\beta}/{4b}}_{n+1}+r^{1/2}_{n+1},
\end{align}
and for $\delta>\beta/2$
\begin{align}\label{induction w}
\|f_{\leq n+1}-f_{\leq n}\|_{C_bB^{1/2-\delta}_{\infty,1}}\leq r^{1/2}_{n+1},
\end{align}
and
\begin{equation}\label{iteration R}
\|q_{n+1}(t)\|_{X}\leq\begin{cases}
r_{n+1},& t\in [2^{-n},\infty),\\
r_{n+1}+\sup_{s\in[t-\ell_{n+1},t]}\|q_{n}(s)\|_{X},&t\in [0,2^{-n}).
\end{cases}
\end{equation}
Consequently, $(f_{\leq n+1},q_{n+1})$ obeys \eqref{inductionv ps},  \eqref{eq:C1}, \eqref{induction wf},  \eqref{eq:R} and \eqref{bd:R} at the level $n+1$.
Moreover, for any $\eps>0$ we could choose the parameter $a$ large enough depending on $M_L, C_0$ such that
\begin{align}\label{induction w1}\|f_{\leq n+1}-f_{\leq n}\|_{C_bB^{1/2-\delta}_{\infty,1}}\leq \eps/2^{n+1}.
\end{align}
\end{proposition}

  Now, we start the iteration by letting $f_{\leq 0}=P_{\leq \lambda_0/3}\Lambda^{-1}\theta_0$. If $\theta_0\in C^\eta$ then we obtain from \eqref{induction ps}
  $$\nabla q_0 \thickapprox\Lambda f_{\leqslant 0}
   \nabla^{\perp} f_{\leqslant 0} +\mathcal{R} \Lambda^{- 1 + \alpha}
   P_{\leq \lambda_0}\xi -\Lambda^{\gamma - 1} \nabla f_{\leq0} ,$$
   which implies (taking $q_{0}$ mean-free)
  $$\|q_0\|_X\lesssim \|\theta_0\|_{C^\eta}^2+\|\theta_0\|_{C^\eta}+L\leq C(N^2+N+L)=M_L.$$
  Moreover,
  $$\|f_{\leq 0}\|_{B^{1/2}_{\infty,1}}\leq CN.$$
  Consequently, $(f_{\leq0},q_{0})$ obeys \eqref{inductionv ps},  \eqref{eq:C1}, \eqref{induction wf}, \eqref{eq:R} and \eqref{bd:R} at the level $0$.

  Let us now state the first main result of this section which is based on  Proposition \ref{p:iteration}. The proof is presented in Section~\ref{s:thm}.

  \bt\label{thm:6.1}
Let $\delta>0$ be arbitrary. For any  $\mathcal{F}$-measurable initial condition $\theta_0\in C^\eta$, $\eta>1/2$ with zero mean $\mathbf{P}$-a.s.
there exists an $\mathcal{F}$-measurable analytically weak solution $\theta$ to \eqref{eq1} with $\theta(0)=\theta_0$ which belongs to  $$L^p_{\rm{loc}}(0,\infty;B_{\infty,1}^{-1/2})\cap C_{b}([0,\infty),B_{\infty,1}^{-1/2-\delta})\cap C^1_{b}([0,\infty),B^{-3/2-\delta}_{\infty,1})\quad \mathbf{P}\text{-a.s. for all } p\in[1,\infty).$$
Moreover, for any $\varepsilon>0$ we can find an $\mathcal{F}$-measurable analytically weak solution $\theta$ such that
\begin{equation}\label{eq:32}
\|\theta\|_{C_b([0,\infty),B^{-1/2-\delta}_{\infty,1})}\leq \|\theta_0\|_{B^{-1/2-\delta}_{\infty,1}}+\varepsilon,\quad
\end{equation}
Furthermore, there exists $m\in\mathbb{N}$ and $\vartheta>1/2$ so that for
\begin{equation}\label{est:fn}
\mathbf{E}\left[\|\theta\|_{C_b([4,\infty),B^{\vartheta-1}_{\infty,1})}+\|\theta\|_{C^1_b([0,\infty),B^{-3/2-\delta}_{\infty,1})}\right]\lesssim 1+\mathbf{E}[\|\theta_0\|_{C^\eta}^m].
\end{equation}
If $\|\theta_0\|_{B^{-1/2-\delta}_{\infty,1}}$ is bounded $\mathbf{P}$-a.s. or if $\theta_{0}$ is non-Gaussian, then the law of the solution $\theta(t)$ is non-Gaussian for any $t\geq0$. There are infinitely many such solutions $\theta$.
\et

The construction of non-Gaussian solutions in case of Gaussian initial conditions is more involved and is carried out in Section~\ref{s:gaus} below.

\subsection{Proof of Proposition~\ref{p:iteration}}
\label{s:it}

\subsubsection{Choice of parameters.}\label{s:p}
In the sequel, we use the following bounds
$$\beta<3/2-\gamma,\quad \frac{1}{b}+\beta<\frac12,\quad (1-\alpha-\kappa)\wedge(\eta-1/2)\wedge (2-\gamma)>b \beta.$$
We can choose $b=4$ so that all the conditions  are satisfied if $\beta$ is small enough. Moreover, we can choose $a$ large enough to absorb implicit constant  and such that \eqref{para} holds.
\subsubsection{Mollification}We intend to replace $f_{\leq n}$ by a mollified one. Let $\{\varphi_\varepsilon\}_{\varepsilon>0}$ be a family of  standard mollifiers with support of $\varphi$ in $(0,1)$. The one sided mollifier here is used in order to preserve initial data. We extend $f_{\leq n}, q_n$ to $t<0$ by taking values at $t=0$. The equation also holds for $t<0$ as $\partial_t f_{\leq n}(0)=0$ by our construction. We define a mollification of $f_{\leq n}$, $q_n$  in  time by convolution as follows
	$$f_{\ell}=f_{\leq n}*_t\varphi_{\ell_{n+1}},\qquad
	q_\ell=q_n*_t\varphi_{\ell_{n+1}},$$
	where  $\ell:=\ell_{n+1}=\lambda_{n+1}^{-1}$.
We apply mollification on  both sides of \eqref{induction ps} and obtain
\begin{align}\label{mol}
\begin{aligned}
- \partial_t \mathcal{R}f_{\ell} + \Lambda f_{\ell} \nabla^{\perp}
   f_{\ell} +\mathcal{R} \Lambda^{- 1 + \alpha}P_{\leq\lambda_{n}} \xi
   &\approx  \Lambda^{\gamma - 1} \nabla f_{\ell} + \nabla q_{\ell} +
   R_{\tmop{com}} ,\\
   f_\ell(0)&=P_{\leq\lambda_n/3}\Lambda^{-1}\theta_0.
   \end{aligned}
   \end{align}
   Here $R_{\tmop{com}}=\Lambda f_{\ell} \nabla^{\perp}
   f_{\ell}-(\Lambda f_{\leq n}\nabla^\bot f_{\leq n})*_t\varphi_{\ell}$.
  In view of \eqref{eq:C1} we obtain for $b(1-\beta/2)-1/2>\vartheta>1/2$
 \begin{align}\label{mol1}
  \begin{aligned}
\|f_{\ell}-f_{\leq n}\|_{C_bC^{\vartheta}}&\lesssim \ell\lambda_n^{\vartheta-1/2}\|f_{\leq n}\|_{C^1_bB_{\infty,1}^{1/2}}
  \\&\lesssim M_0\lambda_{n+1}^{-1} M^{1/2}_L\lambda^{1/2+\vartheta}_n\leq \frac14r^{1/2}_{n+1},
  \end{aligned}
  \end{align}
  and
  \begin{align}\label{mol2}
\|f_\ell\|_{C^1_bB_{\infty,1}^{1/2}}\leq \|f_{\leq n}\|_{C^1_bB_{\infty,1}^{1/2}}\leq M_0M^{1/2}_L\lambda_n,
  \end{align}
and
for $\delta>\beta/2$ by \eqref{induction wf}
\begin{align}
\label{mol3}
\|f_{\ell}\|_{C^1_{b}B^{-1/2-\delta}_{\infty,1}}\leq\|f_{\leq n}\|_{C^1_{b}B^{-1/2-\delta}_{\infty,1}}\leq M_L^{1/2}+M_0M^{1/2}_L\sum_{k=1}^n(\lambda_{k}^{-\delta}+\lambda_{k-1}^{-3/2-\eta}).
\end{align}

\subsubsection{Construction of $f_{\leq n+1}$}\label{sec:conf}

We first introduce a smooth cut-off function satisfying
\begin{align*}
\tilde{\chi}(t)=\begin{cases}
r_n,&t\geq 2^{-n+2},\\
\in (r_n,3M_L),& t\in (3\cdot2^{-n},2^{-n+2} ),\\
3M_L,&t\leq 3\cdot2^{-n}.
\end{cases}
\end{align*}
Here in the middle interval we  smoothly interpolate such that
$\|\tilde{\chi}'\|_{C^{0}}\leq 3M_L2^{n}$ and in view of \eqref{eq:R}, \eqref{bd:R} and since $\supp\varphi_\ell\subset[0,\ell],$ $ \ell<2^{-n}$, it holds $\|q_\ell\|_X\leq\tilde{\chi}$.
Now, similarly to \cite[(2.10)]{CKL21} we define
$$
\tilde{f}_{n+1}(t,x)=\sum_{j=1}^2a_{j,n+1}(t,x)\cos (5\lambda_{n+1}l_j\cdot x), \quad a_{j,n+1}=2\sqrt{\frac{\tilde{\chi}}{5\lambda_{n+1}}}P_{\leq \mu_{n+1}}\sqrt{C_0+\mathcal{R}_j^o\frac{q_{\ell}}{\tilde{\chi}}},
$$
where $l_1=(\frac35,\frac45)^T$, $l_2=(1,0)^T$.
As $L,N$  as well as the parameters $a,b,\beta$  are deterministic and $q_\ell$ is $\mathcal{F}$-measurable,   $\tilde{f}_{n+1}$ is $\mathcal{F}$-measurable.

\begin{remark}\label{r:1}
We note that since $\mu_{n+1}$ is much smaller than $\lambda_{n+1}$, the spatial frequencies of $\tilde{f}_{n+1}$ are localized to $\lambda_{n+1}$. Its Fourier coefficients take the form for $k\in \mathbb{Z}^{2}$
\[
\begin{aligned}
\widehat{\tilde{f}_{n + 1}} (k)
 &=\begin{cases}
\frac1{(2\pi)^2}\sum_{j = 1}^2 \sqrt{\frac{\tilde{\chi}}{5 \lambda_{n + 1}}} \left\langle
   \sqrt{C_0 +\mathcal{R}_j^o \frac{q_{\ell}}{\tilde{\chi}}}, P_{\leq\mu_{n+1}}e^{-i (k+5 \lambda_{n + 1}
   l_j ) \cdummy x} \right\rangle , &\text{if } | k+5 \lambda_{n + 1} l_j  | \leqslant \mu_{n + 1}, \\
\frac1{(2\pi)^2} \sum_{j = 1}^2\sqrt{\frac{\tilde{\chi}}{5 \lambda_{n + 1}}} \left\langle
   \sqrt{C_0 +\mathcal{R}_j^o \frac{q_{\ell}}{\tilde{\chi}}}, P_{\leq\mu_{n+1}} e^{-i (k - 5 \lambda_{n +
   1} l_j) \cdummy x} \right\rangle , &\text{if } | k - 5 \lambda_{n + 1} l_j | \leqslant \mu_{n + 1},\\
 0, &\text{otherwise}.
   \end{cases}
   \end{aligned}
   \]
This is what we employ frequently throughout the paper.
\end{remark}

We also introduce the following cut-off function
\begin{align*}
\chi(t)=\begin{cases}
0,& t\leq 2^{-n-1},\\
\in (0,1),& t\in (2^{-n-1},2^{-n} ),\\
1,&t\geq 2^{-n}.
\end{cases}
\end{align*}
Here in the middle interval we smoothly interpolate so that it holds
$\|\chi'\|_{C^{0}}\leq 2^{n+1}$. Now, define $$f_{n+1}=\chi\tilde{f}_{n+1}+f_{n+1}^{in}:=\chi\tilde{f}_{n+1}+(P_{\leq \lambda_{n+1}/3}-P_{\leq \lambda_n/3})\Lambda^{-1}\theta_0,$$
which is mean zero and $\mathcal{F}$-measurable. We first have for $1/2<\vartheta<\eta$
\begin{align}\label{ini}\|f_{n+1}^{in}\|_{B_{\infty,1}^{1/2}}\lesssim \|f_{n+1}^{in}\|_{C^\vartheta}\lesssim\lambda_n^{-1}\|\theta_0\|_{C^\eta}\leq \frac12r^{1/2}_{n+1},
\end{align}
and
\begin{align}\label{ini1}\|f_{n+1}^{in}\|_{B_{\infty,1}^{-1/2-\delta}}\lesssim \lambda_n^{{-3/2-\eta}}\|\theta_0\|_{C^\eta}\leq M_L^{1/2}\lambda_n^{{-3/2-\eta}},
\end{align}
where we used $\beta b<2$.
In view of the definition of $\tilde{f}_{n+1}$, we obtain
\begin{align}\label{fn}
\|\tilde{f}_{n+1}(t)\|_{B_{\infty,1}^{1/2}}&= \sum_j2^{j/2}\|\Delta_j\tilde{f}_{n+1}\|_{L^\infty}\lesssim\sum_{j:2^j\thicksim \lambda_{n+1}}2^{j/2}\|\Delta_j\tilde{f}_{n+1}\|_{L^\infty}\lesssim \lambda^{1/2}_{n+1}\|\tilde f_{n+1}\|_{L^{\infty}}\nonumber
\\&\lesssim \lambda_{n+1}^{1/2}\left((C_0+1)\frac{\tilde\chi}{5\lambda_{n+1}}\right)^{1/2}\leq
\begin{cases}
M_0r_n^{1/2},& t\in[2^{-n+2},\infty),\\
 M_{0}M_{L}^{1/2},&t\in [0,2^{-n+2}).
 \end{cases}
\end{align}
Here we used the following argument in the third step: By $2^j\sim \lambda_{n+1}$ there exists $c_{1},c_{2}>0$ such that $c_{1}\lambda_{n+1}\leq 2^j\leq c_{2}\lambda_{n+1}$. Hence $\log_{2} c_{1}+\log_{2} \lambda_{n+1}\leq j\leq\log_{2} c_{2}+\log_{2} \lambda_{n+1}$ and $$\sum_{j:2^j\sim \lambda_{n+1}}1\lesssim 1.$$
For the $C^\vartheta$-estimate we use $r_n=M_L\lambda_0^\beta \lambda_n^{-\beta}$ to  deduce
\begin{align}
\|\tilde{f}_{n+1}(t)\|_{C^\vartheta}&= \sup_j2^{j\vartheta}\|\Delta_j\tilde{f}_{n+1}(t)\|_{L^\infty}\lesssim\sup_{j:2^j\thicksim \lambda_{n+1}}2^{j\vartheta}\|\Delta_j\tilde{f}_{n+1}(t)\|_{L^\infty}\lesssim \lambda^{\vartheta}_{n+1}\|\tilde f_{n+1}(t)\|_{L^{\infty}}\nonumber
\no\\&\lesssim \lambda_{n+1}^{\vartheta}\left((C_0+1)\frac{\tilde\chi(t)}{5\lambda_{n+1}}\right)^{1/2}\leq
\begin{cases}
M_0M_{L}^{1/2}\lambda_0^{\beta/2}\lambda_{n+1}^{\vartheta-1/2-{\beta}/{2b}},& t\in[2^{-n+2},\infty),\\
 M_{0}M_{L}^{1/2}\lambda_{n+1}^{\vartheta-1/2},&t\in [0,2^{-n+2}).
 \end{cases}\label{est:fn1}
\end{align}
Furthermore, for $t\geq0$ we obtain
\begin{align}\label{fn3}
\|\tilde{f}_{n+1}(t)\|_{B_{\infty,1}^{1/2-\delta}}&= \sum_j2^{j(1/2-\delta)}\|\Delta_j\tilde{f}_{n+1}(t)\|_{L^\infty}
\\&\lesssim \lambda_{n+1}^{1/2-\delta}\left((C_0+1)\frac{M_L}{5\lambda_{n+1}}\right)^{1/2}\leq \frac14r^{1/2}_{n+1}.\nonumber
\end{align}
Here we used $\delta>\beta/2$ and we chose $a$ sufficiently large depending on $C_{0}$ to absorb the constant.

 The new solution $f_{\leq n+1}$ is defined as
$$f_{\leq n+1}:=f_\ell+f_{n+1},$$
which is also mean zero and $\mathcal{F}$-measurable.

\subsubsection{Inductive estimates for $f_{\leq n+1}$}
From the construction we see that $\supp\widehat{f_{\leq n+1}}\subset \{|k|\leq 6\lambda_{n+1}\}$ and $f_{\leq n+1}(0)=P_{\leq \lambda_{n+1}/3}\Lambda^{-1}\theta_0$. We first prove \eqref{iteration ps} and \eqref{induction fn}. Combining \eqref{ini}, \eqref{mol1}, \eqref{fn} and \eqref{est:fn1} we obtain for $t\in[2^{-n+2},\infty)$
$$\|f_{\leq n+1}(t)-f_{\leq n}(t)\|_{B_{\infty,1}^{1/2}}\leq \|f_\ell-f_{\leq n}\|_{C_tB_{\infty,1}^{1/2}}+\|f_{n+1}(t)\|_{B_{\infty,1}^{1/2}}\leq M_0r_n^{1/2}+r^{1/2}_{n+1},$$
and for $\vartheta<(\frac12+\frac{\beta}{4b})\wedge \eta$
$$\|f_{\leq n+1}(t)-f_{\leq n}(t)\|_{C^{\vartheta}}\leq \|f_\ell-f_{\leq n}\|_{C_tC^\vartheta}+\|f_{n+1}(t)\|_{C^\vartheta}\leq M_0M_L^{1/2}\lambda_{0}^{\beta/2}\lambda^{-\beta/4b}_{n+1}+r^{1/2}_{n+1},$$
which implies \eqref{induction fn}.
Also combining \eqref{ini} and \eqref{mol1} and \eqref{fn} we obtain for $t\in(2^{-n-1},2^{-n+2})$
$$\|f_{\leq n+1}(t)-f_{\leq n}(t)\|_{B_{\infty,1}^{1/2}}\leq \|f_\ell-f_{\leq n}\|_{C_tB_{\infty,1}^{1/2}}+\|f_{n+1}(t)\|_{B_{\infty,1}^{1/2}}\leq M_0M_L^{1/2}+r^{1/2}_{n+1}.$$
Also combining \eqref{ini} and \eqref{mol1}  we obtain for $t\in[0, 2^{-n-1}]$
$$\|f_{\leq n+1}(t)-f_{\leq n}(t)\|_{B_{\infty,1}^{1/2}}\leq \|f_\ell-f_{\leq n}\|_{C_tB_{\infty,1}^{1/2}}+\|f_{n+1}(t)\|_{B_{\infty,1}^{1/2}}\leq r^{1/2}_{n+1}.$$
 Thus \eqref{iteration ps} holds. Moreover, we have for $t\geq0$
\begin{align*}
\|f_{\leq n+1}(t)\|_{B_{\infty,1}^{1/2}}&\leq N+\sum_{k=0}^n\|f_{\leq k+1}(t)-f_{\leq k}(t)\|_{B_{\infty,1}^{1/2}}
\\&\leq N+\sum_{k=0}^n(M_0r_k^{1/2}+r^{1/2}_{k+1})+\sum_{k=0}^nM_0M_L^{1/2}1_{t\in (2^{-k-1},2^{-k+2}]}
\\&\leq N+\sum_{k=0}^n(M_0r_k^{1/2}+r^{1/2}_{k+1})+3M_0M_L^{1/2}.
\end{align*}
Thus by \eqref{para}, \eqref{inductionv ps} follows for $f_{\leq n+1}$. \eqref{induction w} and \eqref{induction w1} follow from \eqref{fn3}, \eqref{ini} and \eqref{mol1}. Now, we estimate the $C^1_bB^{1/2}_{\infty,1}$ norm. With \eqref{inductionv ps} at hand, we shall estimate the time derivative. We use the definition of $\tilde\chi$ and \eqref{eq:R}, \eqref{bd:R} to obtain
\begin{align}\label{fn2}
	\|\partial_{t}\tilde{f}_{n+1}\|_{C_bB_{\infty,1}^{1/2}}&\leq \sup_{t\in[0,\infty)}\sum_{j:2^j\sim \lambda_{n+1}}2^{j/2}\|\Delta_j\partial_t\tilde{f}_{n+1}\|_{L^\infty}
	\\&\lesssim \lambda_{n+1}^{1/2}\left[\left(\frac{M_L}{5\lambda_{n+1}}\right)^{1/2}(\ell^{-1}+3M_L2^nr_n^{-1})+3M_L2^n(C_{0}+1)^{
	{1/2}}(\lambda_{n+1}r_n)^{-1/2}\right]\nonumber
	\\&\leq (M_0-2)M^{1/2}_L\lambda_{n+1}.\nonumber
\end{align}
Here we used that $2^n\lambda_n^\beta\leq \lambda_{n+1}$. The time derivative of $\chi$ could be bounded by $2^{n+1}\leq\lambda_{n+1}$. Hence  \eqref{eq:C1} holds at level of $n+1$. Similarly we have
$$\|\chi\tilde{f}_{n+1}\|_{C^1_bB_{\infty,1}^{-1/2-\delta}}\leq M_0M_L^{1/2}\lambda_{n+1}^{-\delta},$$
Thus \eqref{induction wf} follows from \eqref{ini1} and \eqref{mol3} at the level of $n+1$.

\subsubsection{Inductive estimate for $q_{n+1}$}\label{s:estq}
Subtracting from \eqref{induction ps} at level $n+1$ the system \eqref{mol}, we obtain
\begin{align*}
\nabla q_{n+1}\approx& \underbrace{\chi^2\Lambda \tilde{f}_{n+1}\nabla^\bot \tilde{f}_{n+1}+\nabla q_\ell}_{\nabla q_M}+\underbrace{\Lambda f_{n+1}\nabla^\bot f_{\ell}+\Lambda f_{\ell}\nabla^\bot f_{ n+1}}_{\nabla q_T}\underbrace{- \nabla \Lambda^{\gamma-1}f_{n+1}}_{\nabla q_D}
	\\&\underbrace{-\partial_t\mathcal{R}f_{n+1}+\Lambda(\chi\tilde{f}_{n+1})\nabla^\perp f_{n+1}^{in}+\nabla^\perp(\chi\tilde{f}_{n+1})\Lambda f_{n+1}^{in}+\Lambda f_{n+1}^{in}\nabla^\perp f_{n+1}^{in}}_{\nabla q_I}
	\\&+R_{\rm{com}}\underbrace{-\mathcal{R} \Lambda^{- 1 + \alpha} (P_{\leq\lambda_{n}}\xi - P_{\leq\lambda_{n+1}}\xi) }_{\nabla q_N}.
\end{align*}
As the left hand side is mean free, the {right} hand side is also mean free. Then for each term on the right hand side we could subtract its mean part and in the following we do not change the notation for simplicity. For the mean free part we could define the inverse of $\nabla$ by $\Delta^{-1}\nabla\cdot$.

As a consequence of the above  formula, we deduce that  $\supp\hat{q}_{n+1}\subset\{|k|\leq 12\lambda_{n+1}\}.$ First, we estimate each term  for $t\in [2^{-n},\infty)$. We follow the same arguments as  in the proof of Proposition~3.1 in \cite{CKL21} with $r_n$ replaced by $\tilde\chi$  to have
$$\|q_M\|_X\lesssim \log \mu_{n+1}(\mu_{n+1}^{-1}\lambda_n)^2\tilde{\chi}+\log \mu_{n+1}(\mu_{n+1}\lambda_{n+1}^{-1})^2\tilde{\chi}+\lambda_{n}\lambda_{n+1}^{-1}\tilde{\chi}\leq\frac16r_{n+1},
$$
where we used that $1>1/b+\beta$. Here and in the following the implicit constants  depend on $M_0$ and we choose $a$ large enough to absorb them. Now we consider $q_T$ and have
$$\nabla q_T=\chi\Lambda \tilde{f}_{n+1}\nabla^\bot f_{{\ell}}+\chi\Lambda f_{{\ell}}\nabla^\bot \tilde{f}_{ n+1}+\Lambda f^{in}_{n+1}\nabla^\bot f_{{\ell}}+\Lambda f_{{\ell}}\nabla^\bot f^{in}_{ n+1}.$$
Similarly as \cite{CKL21} and using the definition of $\tilde{f}_{n+1}$ and \eqref{inductionv ps} we have for the first two terms
\begin{align*}
	&\Big\|\Delta^{-1}\nabla\cdot\Big(\chi\Lambda \tilde{f}_{n+1}\nabla^\bot f_{{\ell}}+\chi\Lambda f_{{\ell}}\nabla^\bot \tilde{f}_{ n+1}\Big)\Big\|_X
	\\&\leq\log\lambda_{n+1} \|\tilde f_{n+1}\|_{L^\infty}(\|\Lambda f_{{\ell}}\|_{L^\infty}+\|\nabla^\bot f_{{\ell}}\|_{L^\infty})\lesssim \log\lambda_{n+1}M^{1/2}_L\lambda_n^{1/2}(\tilde{\chi}\lambda_{n+1}^{-1})^{1/2}<\frac1{12}r_{n+1}.
\end{align*}
Here we used $1/2>1/({2b})+\beta.$
Using \eqref{inductionv ps}, \cite[Lemma 3.2]{CKL21} and $\|\theta_0\|_{C^\eta}\leq N$, we have for the last two terms
$$
\Big\|\Delta^{-1}\nabla\cdot\Big(\Lambda f^{in}_{n+1}\nabla^\bot f_{\ell}+\Lambda f_{\ell}\nabla^\bot f^{in}_{ n+1}\Big)\Big\|_X\lesssim N\lambda_n^{-\eta}M^{1/2}_L\lambda_n^{1/2}\log \lambda_{n+1}<\frac1{12}r_{n+1},$$
provided that $\eta-1/2> b\beta$. Here, we also understand these two terms as their mean free part only.

For $q_D$ we  use the support of the Fourier transform of $\tilde f_{n+1}$ and $f_{n+1}^{in}$ contained in an annulus and \eqref{fn} to have
$$
\|q_D\|_X\lesssim \lambda_{n+1}^{\gamma-1}\sqrt{M_L\lambda_{n+1}^{-1}}+\lambda_n^{\gamma-2}N<\frac16r_{n+1}.$$
Here we used $\beta<3/2-\gamma$ and $2-\gamma>b\beta$.
For $q_N$ we have
\begin{align*}
	\| q_N\|_{X} \lesssim L\lambda^{-1 +\alpha+\kappa}_n<\frac16 r_{n+1}, \end{align*}
provided that $1-\alpha-\kappa>\beta b$.
It remains to estimate $q_{I}$. As the initial data part is independent of time, the first term in $q_I$ is bounded as in \eqref{fn2} by
$$
\begin{aligned}
	&\|\Delta^{-1}\nabla \cdot \partial_t\mathcal{R} \tilde f_{n+1}\|_{X}\\
	&\qquad\lesssim  \frac{\log\lambda_{n+1}} {\lambda_{n+1}} \left[\left(\frac{M_L}{5\lambda_{n+1}}\right)^{1/2}(\ell^{-1}+3M_L2^nr_n^{-1})+3M_L2^n(C_{0}+1)^{{1/2}}(\lambda_{n+1}r_n)^{-1/2}\right]\\
	&\qquad<\frac1{24}r_{n+1}.
\end{aligned}
$$
Using \eqref{fn} and the fact that the Fourier support of $\Lambda(\chi\tilde{f}_{n+1})\nabla^\perp f_{n+1}^{in}+\nabla^\perp(\chi\tilde{f}_{n+1})\Lambda f_{n+1}^{in}$ is contained in an annulus of radius $\lambda_{n+1}$, and applying the regularity of the initial condition in the estimate of   $\Lambda f_{n+1}^{in}\nabla^\perp f_{n+1}^{in}$ (i.e. its mean free part),  the other terms are bounded by
$$
\log \lambda_{n+1}N\sqrt{\frac{{M_L}}{\lambda_{n+1}}}+N^2\lambda_n^{-2\eta}{\log \lambda_{n+1}}<\frac1{12}r_{n+1},$$
provided that ${\eta>b \beta/2}$.
Moreover, using \eqref{inductionv ps}, \eqref{eq:C1} we get
\begin{align*}
\|R_{\rm{com}}\|_{X}&\lesssim \log \lambda_{n}\ell_{n+1}\|\Lambda f_{\leq n}\nabla^\bot f_{\leq n}\|_{C^1L^\infty}+\log \lambda_{n}\ell_{n+1}\|f_{\leq n}\|_{CC^1}\|f_{\leq n}\|_{C^1C^1}
	\\&\lesssim\log \lambda_{n} \ell_{n+1}\lambda^2_nM_L\leq\frac16 r_{n+1},
\end{align*}
which means the corresponding term in $q_{n+1}$ could be controlled by Poincar\'e's inequality.
Here we used that $1/b+\beta/2<1/2$.

For $t\in [2^{-n-1},2^{-n} ]$ we have the extra term $\mathcal{R}\tilde{f}_{n+1}\partial_t\chi$ and
$q_{n+1}^1=(1-\chi^2)q_\ell.$
Then we have
$$\|q_{n+1}^1(t)\|_X\leq\sup_{s\in[t-\ell,t]}\|q_n(s)\|_X,$$
as well as
$$\|\Delta^{-1}\nabla \cdot \mathcal{R} \tilde f_{n+1}\partial_t\chi\|_{{X}}\lesssim 2^{n+1}\lambda_{n+1}^{-3/2}M_L^{1/2}\log\lambda_{n+1}<\frac1{24}r_{n+1}.$$
Moreover on $t\in[0,2^{-n-1})$ we have
$$\nabla q_{n+1}\approx\nabla q_\ell+R_{\rm{com}}-\nabla \Lambda^{\gamma-1}f^{in}_{n+1}+\nabla q_N+\Lambda f^{in}_{n+1}\nabla^\bot f_{{\ell}}+\Lambda f_{{\ell}}\nabla^\bot f^{in}_{ n+1}+\Lambda f_{n+1}^{in}\nabla^\perp f_{n+1}^{in}.$$
Except for the first term, the other ones could be bounded by the previous argument. Thus we deduce
$$\|q_{n+1}(t)\|_X\leq r_{n+1}+\sup_{s\in[t-\ell,t]}\|q_n(s)\|_X.$$
Hence \eqref{iteration R}, \eqref{eq:R} and  \eqref{bd:R} hold at level of $n+1$ and the proof of Proposition~\ref{p:iteration} is complete.

\subsection{Proof of Theorem \ref{thm:6.1}}\label{s:thm}

Let the additional assumption  \eqref{eq:u0} be satisfied  for some $L, N\geq 1$. We repeatedly apply Proposition~\ref{p:iteration} and obtain a sequence of  $\mathcal{F}$-measurable processes $(f_{\leq n},q_{n})\in C_0^\infty\times C_0^\infty$, $n\in\N_{0}$, such that  $f_{\leq n} \to f$ in
$C_{b}([0,\infty);B^{1/2-\delta}_{\infty,1})$  as a consequence of  \eqref{induction w}. Moreover, using {\eqref{iteration ps}} we
have for every $p\in[1,\infty)$
\begin{align*}
  & \int_0^{T}  \| f_{\leq n + 1} - f_{\leq n} \|_{B^{1/2}_{\infty,1}}^p \dif t
\\ &\leq\int_{2^{-n+2}}^{T} \| f_{\leq n + 1} - f_{\leq n}\|_{B^{1/2}_{\infty,1}}^p \dif t +
   \int_{2^{-n-1}}^{2^{-n+2}} \| f_{\leq n + 1} - f_{\leq n} \|_{B^{1/2}_{\infty,1}}^p \dif t
   +\int_0^{2^{-n-1}} \| f_{\leq n + 1} - f_{\leq n} \|_{B^{1/2}_{\infty,1}}^p \dif t
\\&\lesssim 2^{- n} M_L^{p/2}  +( r_{n+1}^{1/2}+r_n^{1/2})^{p}.
\end{align*}
Thus, the sequence $f_{\leq n}$, $n\in\N_{0}$, is  Cauchy hence converging in $L^{p}(0,T;B_{\infty,1}^{1/2})$ for all $p\in[1,\infty)$.
Accordingly, $f_{\leq n}\to f$ also in $L^p_{\rm{loc}}(0,\infty;B_{\infty,1}^{1/2})$.
Furthermore, by \eqref{eq:R}, \eqref{bd:R} we know for all $p\in[1,\infty)$
	$$
	\int_0^{T}\|q_n(t)\|^{p}_{X}\dif t\leq r_{n}^{p}T+(3M_L)^{p}2^{-n+1} \to 0, \quad\mbox{as}\quad n\to\infty.
	$$
Thus, by \eqref{besov} and similarly as in  \cite{CKL21} the process $\theta=\Lambda f$ satisfies \eqref{eq1}  in the analytically weak sense.
More precisely, we define $\theta_n=\Lambda f_{\leq n}$ and from \eqref{induction ps} we obtain for any $\psi\in C^\infty$
\begin{align*}
\langle \theta_n(t)-P_{\leq {\lambda_n/3}}\theta_0,\psi\rangle= \int_0^t\langle \theta_n
   \nabla^{\perp} \Lambda^{-1}\theta_n+\nabla q_n, \nabla\psi\rangle + \langle\Lambda^{\alpha}
   P_{\leq \lambda_n}\xi -\Lambda^{\gamma } \theta_{n}
,\psi\rangle\dif s.
\end{align*}
We rewrite the above as
\begin{align*}
\langle \theta_n(t)-P_{\leq {\lambda_n/3}}\theta_0,\psi\rangle= \int_0^t-\frac12\langle \Lambda^{-1/2}\theta_n,\Lambda^{1/2}[\mathcal{R}^\perp,\nabla \psi] \theta_n\rangle
  +\langle \nabla q_n, \nabla\psi\rangle + \langle\Lambda^{\alpha}
   P_{\leq \lambda_n}\xi -\Lambda^{\gamma } \theta_{n}
,\psi\rangle\dif s.\end{align*}
Let $n\rightarrow\infty$. As $\theta_n\rightarrow \theta$  in $L^{p}_{\text{loc}}(0,\infty;B_{\infty,1}^{-1/2})\cap C_{b}([0,\infty);B^{-1/2-\delta}_{\infty,1})$ and $q_n\to0$ in $L^p_{\text{loc}}([0,\infty);L^\infty)$ for all $p\in[1,\infty)$  by \cite[Proposition 5.1]{CKL21} and \eqref{besov} we deduce that $\theta$ satisfies \eqref{eq2} in the sense of Definition~\ref{d:1}. By \eqref{induction wf} and lower-semicontinuity we obtain
\begin{align}
\label{c1}
\|\theta\|_{C^1_{{b}}B^{-3/2-\delta}_{\infty,1}}\leq M_L^{1/2}+6M_0M_L^{1/2}.\end{align}
Also \eqref{induction w1} implies \eqref{eq:32}.

Next, we show that
\eqref{est:fn} follows from \eqref{induction fn}. Indeed, from the construction we know that the parameters $\beta,\delta$ are independent of $M_L$. Only \eqref{induction w1} requires
$a$ large enough depending on $M_L$. In view of \eqref{induction w} we could choose $a$ satisfying $a^{\beta}\simeq M_L^{1/2}/\eps$. By \eqref{induction fn} and \eqref{para} we know for some $\vartheta>1/2$
\begin{equation}
\label{eq:ee}
\|\theta\|_{C_{{b}}([4,\infty),B^{\vartheta-1}_{\infty,1})}\leq 3M_0M_L^{1/2}a^{\beta/4}+3M_0M_L^{1/2}+\|\theta_0\|_{C^\eta},
\end{equation}
Here $\vartheta$ might be different from $\vartheta$ in \eqref{induction fn} and by \eqref{induction fn} we could always find such $\vartheta$.

For a general initial condition $\theta_{0}\in C^\eta$ $\mathbf{P}$-a.s., define $$\Omega_{N,L}:=\{N-1\leq \|\theta_{0}\|_{C^\eta}< N\}\cap \{L-1\leq \|\xi\|_{C^{-1-\kappa}}< L\}\in\mathcal{F}.$$ Then the first part of this proof gives the existence of infinitely many measurable solutions $\theta^{N,L}$ on each $\Omega_{N,L}$. Letting $\theta:=\sum_{N,L\in\mathbb{N}}\theta^{N,L}1_{\Omega_{N,L}}$ we obtain a solution defined on the whole $\Omega$. Finally, using the moments of $\theta_{0}$ and the noise $\xi$ we obtain \eqref{est:fn} from \eqref{eq:ee} and \eqref{c1}, which completes the first part of the proof.

\subsubsection{Non-Gaussianity}
We first prove that the solution is not Gaussian if $\theta_0$ is bounded in $B^{-1/2-\delta}_{\infty,1}$. We can choose $a$ large enough such that \eqref{induction w1} holds and \eqref{eq:32} follows. From this, we now  prove that the solution is not Gaussian. If there exists $k \in \mathbb{Z}^2$ so that the Fourier coefficient $ \hat{f}(k)$ is not a.s. equal to a constant, then we are done. Namely, if $\theta_0$ is bounded in $B^{-1/2-\delta}_{\infty,1}$ then  \eqref{eq:32}  implies
that $\hat{f}(k)$ is bounded a.s. but not constant a.s.,  hence it is not Gaussian. Therefore also $f$ is not Gaussian.
If such a $k$ does not exist, then $f$ is deterministic so it cannot solve the
equation and we have a contradiction.

    If $\theta_0$ is non-Gaussian, let $k \in \mathbb{Z}^2$ be such that
$\hat{\theta}_0 (k)$ is non-Gaussian. Then we run each convex integration on
$\Omega_{N, L}$ with the additional condition (uniform in $\omega$)
\[ \lambda_0 / 3 \geqslant | k | . \]
Consequently, the perturbations $f^{N,L}_{n+1}$, $n\in\mathbb{N}_{0}$, do not affect $\hat{f}(k)$ and it holds
$$\hat{f} (k) = \sum_{N,L}\widehat{f^{N,L}_{\leqslant 0}} (k)1_{\Omega_{N,L}} =\mathcal{F}
[P_{\leqslant \lambda_0 / 3} \Lambda^{- 1} \theta_0] (k) =\mathcal{F}
[\Lambda^{- 1} \theta_0] (k)$$ which is non-Gaussian.

\subsubsection{Non-uniqueness of solutions}\label{s:nu}
Non-uniqueness can be proved by a small modification of our construction, the idea comes from \cite{Ise22}. Namely, if we  switch
 the sign of one $a_{j, n + 1}$ for one $j\in\{1,2\}$ during our construction, then all the estimates in Proposition \ref{p:iteration}
do not change and we obtain different solutions. To be more precise, let $\sigma:\mathbb{N}\rightarrow \{1,2\}$.
Let us take for $n\in\mathbb{N}_0$
$$a^\sigma_{1, n + 1} = (-1)^{\sigma(n+1)} a_{1, n + 1}, \qquad a^\sigma_{2, n + 1} = a_{2, n +
   1} $$
and define $f^\sigma_{n + 1}$  using $a_{j,n+1}^\sigma$ in place of $a_{j,n+1}$. Then we obtain that $f^\sigma_{n+1}$ has a limit $f^\sigma$ and $\theta^\sigma=\Lambda f^\sigma$  is a solution to \eqref{eq1}. We show that for $\sigma_1\neq \sigma_2$ the corresponding solutions $f^{\sigma_1}$, $ f^{\sigma_2}$ are different. Let $n$ be the smallest number such that $\sigma_1(n)\neq \sigma_2(n)$. Without loss of generality we assume $\sigma_{1}(n)=2$, $\sigma_{2}(n)=1$. For $t\geq4$ it holds
$$ f^{\sigma_1}_{n }(t,x) - f^{\sigma_2}_{n }(t,x) = 2 a_{1, n } (t,x)\cos (5 \lambda_{n } l_1
   \cdummy x). $$
Now, similarly as in Remark~\ref{r:1}, we can compute the $k$th Fourier
coefficient of $f^{\sigma_1}_{n }(t) - f^{\sigma_2}_{n }(t)$ for $t\geq4$. It equals to
\begin{align*}
\mathcal{F}[f^{\sigma_1}_{n }(t) - f^{\sigma_2}_{n }(t)](k)&= \frac2{(2\pi)^2} \sqrt{\frac{r_{n-1}}{5 \lambda_{n }}} \left\langle \sqrt{C_0
   +\mathcal{R}_1^o \frac{q_{\ell_n}}{r_{n-1}}}, P_{\leqslant \mu_{n }} e^{-i (k+5
   \lambda_{n } l_1 ) \cdummy x} \right\rangle \\
&\qquad +  \frac2{(2\pi)^2}\sqrt{\frac{r_{n-1}}{5 \lambda_{n }}} \left\langle \sqrt{C_0
   +\mathcal{R}_1^o \frac{q_{\ell_n}}{r_{n-1}}}, P_{\leqslant \mu_{n}} e^{-i (k - 5
   \lambda_{n } l_1) \cdummy x} \right\rangle. \end{align*}
Here $q_{\ell_n}=q_{n-1}*_t\varphi_{\ell_n}$.
The above is non-zero if either
$$ | k+ 5 \lambda_{n} l_1 | \leqslant \mu_{n }, \qquad \tmop{or}
   \qquad | k - 5 \lambda_{n } l_1 | \leqslant \mu_n . $$
   This cannot happen at the same time. Let $g:=2\sqrt{\frac{r_{n-1}}{5 \lambda_{n }}} \sqrt{C_0
   +\mathcal{R}_1^o \frac{q_{\ell_n}}{r_{n-1}}}$.
Then we obtain by Parseval's equality and \eqref{eq:R} for $t\geq4$
\begin{align*}
\int_{\mathbb{T}^{2}} | f^{\sigma_1}_{n }(t) - f^{\sigma_2}_{n }(t)|^2 \dif x&=(2\pi)^2\sum_{k\in\mathbb{Z}^{2}} |
   \mathcal{F}[ f^{\sigma_1}_{n }(t) - f^{\sigma_2}_{n }(t) ](k ) |^2
\\&\geq (2\pi)^2\sum_{k = 5 \lambda_{n } l_1, -5 \lambda_{n } l_1} |
   \mathcal{F}[ f^{\sigma_1}_{n }(t) - f^{\sigma_2}_{n }(t) ](k ) |^2
\\&  = 2\cdot(2\pi)^2| \mathcal{F} g (0) |^2 = \frac2{(2\pi)^2}\left| \int_{\mathbb{T}^{2}}
   g \,\dif x \right|^2 \geqslant \frac{32r_{n-1}\pi^2}{5 \lambda_{n }} (C_0 -
   1). \end{align*}
By \eqref{mol1} and a similar calculation as in \eqref{fn} we also have
   \begin{align*}
   &\| (f^{\sigma_1}_{\leq k + 1}(t) - f^{\sigma_1}_{\leq k}(t)) -
   (f^{\sigma_2}_{\leq k + 1}(t) -f^{\sigma_2}_{\leq k }(t) )\|_{ L^2}
   \\&\leq\| (f^{\sigma_1}_{\ell}(t) - f^{\sigma_1}_{\leq k}(t))\|_{L^2} +
  \| (f^{\sigma_2}_{\ell}(t) -f^{\sigma_2}_{\leq k }(t) )\|_{ L^2}+\| \tilde{f}^{\sigma_1}_{ {k + 1}}(t)\|_{L^2}+\| \tilde{f}^{\sigma_2}_{ {k + 1}}(t)\|_{L^2}
  \\&\leq 2M_0\lambda_{k+1}^{-1}M_L^{1/2}\lambda_k^{1/2}+C\sqrt{r_k/\lambda_{k+1}}\leq C\sqrt{r_k/\lambda_{k+1}}.
   \end{align*}
   Here we used similar argument as \eqref{mol1} and definition of $\tilde f_{n+1}$.
Thus, we obtain for $t\geq4$
\begin{align}
\|f^{\sigma_1}(t) - f^{\sigma_2}(t) \|_{L^2} &\geqslant \| f^{\sigma_1}_{n }(t) - f^{\sigma_2}_{n }(t)
   \|_{L^2} - \sum_{k = n}^{\infty} \| (f^{\sigma_1}_{\leq k + 1}(t) - f^{\sigma_1}_{\leq k}(t)) -
   (f^{\sigma_2}_{\leq k + 1}(t) -f^{\sigma_2}_{\leq k }(t) )\|_{ L^2} \nonumber\\
&\geq   \pi\sqrt{(C_0-1)\frac{32r_{n-1}}{5\lambda_n}}-C\sum_{k=n}^\infty \sqrt{\frac{r_{k}}{\lambda_{k+1}}}.\label{eq:lb}
\end{align}
  By choosing $a$ large enough such that $\lambda_n$ is increasing fast and $C_0$ large enough, we conclude that  the above is positive and consequently the two solutions are different. Thus non-uniqueness  follows.

\subsubsection{Infinitely many solutions}

The result of Section~\ref{s:nu} readily implies the existence of at least
countably many solutions. Assume for a contradiction that there is only a
finite number $M$ of solutions. Then we can find $M + 1$ different solutions
as follows. Let $\sigma_1, \ldots, \sigma_{M + 1} : \mathbb{N} \rightarrow \{
1, 2 \}$ be pairwise different. We apply the computation of Section~\ref{s:nu}
to each pair $\sigma_i, \sigma_j$, $i \neq j$. This way, we may need to modify
the parameters in convex integration at most $\binom{M + 1}{2}$-times, which
can be done. Thus, the corresponding lower bounds \eqref{eq:lb} are strictly
positive for each couple $\sigma_i, \sigma_j$ and hence the obtained $M + 1$
solutions are pairwise different.

 \subsubsection{Continuum of solutions} The idea  is to prove that choosing $n\geq n_{0}$ for some $n_{0}$ the above lower bound \eqref{eq:lb} is positive uniformly in our parameters $a,b,\beta$. This way we do not need to adjust them anymore and we get (at least) one solution for each $\sigma:\mathbb{N}+n_{0}\to\{1,2\}$. Here $\mathbb{N}+n_{0}$ means  that whenever the two mappings $\sigma\in \{1,2\}^{\mathbb{N}}$ differ from position $n_{0}$ on, then the corresponding solutions are different.

Since by our notation we have
\[ \sqrt{\frac{r_{n - 1}}{\lambda_n}} = a^{b^{n - 1} (- \beta - b) / 2}M_L^{1/2}\lambda_0^{\beta/2} \]
the lower bound \eqref{eq:lb} is of the form
\[ M_L^{1/2}\lambda_0^{\beta/2}C_1 a^{b^{n - 1} (- \beta - b) / 2} - M_L^{1/2}\lambda_0^{\beta/2}C_2 \sum^{\infty}_{k = n} a^{b^k (-
   \beta - b) / 2} \]
   for some universal constants $C_{1}, C_{2}>0$.
We claim and prove below, that for another universal constant $C_3>0$ independent of the parameters $a, \beta, b, n$ it holds
\begin{equation}\label{eq:lb1}
 \sum^{\infty}_{k = n} a^{(b^k - b^n) (- \beta - b) / 2} \leqslant C_3.
\end{equation}

As a consequence,
\[ \sum^{\infty}_{k = n} a^{b^k (- \beta - b) / 2} = a^{b^n (- \beta - b) / 2}
   \sum^{\infty}_{k = n} a^{(b^k - b^n) (- \beta - b) / 2} \leqslant C_3
   a^{b^n (- \beta - b) / 2} \]
and the lower bound \eqref{eq:lb} is bounded from below by
\begin{align*}
 &M_L^{1/2}\lambda_0^{\beta/2}\Big(C_1 a^{b^{n - 1} (- \beta - b) / 2} - C_2 C_3 a^{b^n (- \beta -
   b) / 2}\Big) \\
   &\qquad= M_L^{1/2}\lambda_0^{\beta/2}C_1 a^{b^{n - 1} (- \beta - b) / 2} \left( 1 - \frac{C_2 C_3
   a^{b^n (- \beta - b) / 2}}{C_1 a^{b^{n - 1} (- \beta - b) / 2}} \right).
\end{align*}
The expression on the right hand side  is positive if and only if
\begin{equation}
  \frac{C_2 C_3 a^{b^n (- \beta - b) / 2}}{C_1 a^{b^{n -
  1} (- \beta - b) / 2}} < 1 \quad \Leftrightarrow \quad r^{b^{n - 1} - b^{n -
  2}} < C_4 ,\label{eq:89}
\end{equation}
where $r = a^{b (- \beta - b) / 2} \leqslant 1 / 2$ (since we assumed
$a^{\beta b} \geqslant 4$ and $a \geqslant 1$), $b^{n - 1} - b^{n - 2} =
b^{n - 2} (b - 1) > 0$ and $C_{4}>0$ is a universal constant. Thus,  if we require
\begin{equation}
  \left( \frac{1}{2} \right)^{b^{n - 2} (b - 1)} < C_4, \label{eq:8}
\end{equation}
then the above inequality \eqref{eq:89} holds for all $r$, i.e. we get
positivity of the lower bound \eqref{eq:lb} for all choices of our parameters $a, b, \beta$.
Now, it remains to find the smallest $n = n_0$ so that \eqref{eq:8} holds.

Accordingly,  for every $\sigma_1, \sigma_2 : \mathbb{N}+ n_0 \rightarrow \{ 1, 2 \}$,
$\sigma_1 \neq \sigma_2$ we can prove that the corresponding solutions
$f^{\sigma_1}$ and $f^{\sigma_2}$ are different without the need to modify $a,
b, \beta$. Hence with the same choice of parameters, we get a solution for every
such $\sigma$. The cardinality of the set of such $\sigma$'s is $| \{ 1,
2 \}^{\mathbb{N}} | $ and that is continuum.

It remains to prove that uniformly in $a, b, \beta, n$ \eqref{eq:lb1} holds true. We have
\[ \sum^{\infty}_{k = n} a^{(b^k - b^n) (- \beta - b) / 2} \leqslant
   \sum^{\infty}_{k = n} a^{(b^{k - n} - 1) b^n (- \beta - b) / 2} =
   \sum^{\infty}_{k = 0} a^{(b^k - 1) b^n (- \beta - b) / 2}. \]
For $k \geqslant 2$ we have $b^k - 1 \geqslant k b$ so the above term is bounded by
\[1+ a^{(b - 1) b^n (- \beta - b) / 2} + \sum^{\infty}_{k = 2} a^{k b b^n (-
   \beta - b) / 2} \leqslant 1 + \frac{1}{1 - a^{b b^n (- \beta - b) / 2}}. \]
Since $a^{\beta b} \geqslant 4$ and $a \geqslant 1$ we get
\[ a^{b b^n (- \beta - b) / 2} \leqslant a^{b b^n (- \beta) / 2} \leqslant
   {\left( \frac{1}{2} \right)^{b^n}}  \leqslant 1 / 2, \]
hence
\[ 1 + \frac{1}{1 - a^{b b^n (- \beta - b) / 2}} \leqslant 3 =: C_3 .
\]
The proof is complete.

\subsection{A prescribed terminal value and (non)Gaussianity}\label{s:gaus}

The goal of this section is to address the question of (non)Gaussianity at positive times  of solutions starting from possibly Gaussian initial values. It turns out that the law at a fixed time $T\geq 4$ can be prescribed. In fact, by refining our convex integration construction we establish even a stronger  result. It shows that in addition to an initial condition, a terminal condition may be prescribed as well. In view of the non-Gaussianity result from Theorem~\ref{thm:6.1} which applies to non-Gaussian initial conditions, we can then for instance use the terminal value at time $T\geq 4$ as a new initial condition and repeat the first convex integration construction. This yields solutions which are non-Gaussian for every $t\in[4,\infty)$, independent of the law at the initial time. However, we have many other possibilities how to reiterate the convex integration. So in particular, we may as well prescribe that the solutions become Gaussian at each time of the form $T=4k$, $k\in\mathbb{N}$.

The main result of this section reads as follows.

  \bt\label{thm:6.11}
Let $T\geq 4$ {and $\delta>0$ be arbitrary}. For any    $v\in C^1_{T}C^\eta$  $\mathbf{P}$-a.s. with $\eta>1/2$, an $\mathcal{F}$-measurable  random variable with zero mean $\mathbf{P}$-a.s. such that $v(t)=v(T)$ for $t\in [T-1,T]$ and $\partial_tv(0)=0$,
there exists an $\mathcal{F}$-measurable analytically weak solution $\theta$ to \eqref{eq1} with $\theta(0)=v(0)$, $\theta(T)=v(T)$ and which belongs to
\begin{equation}\label{eq:tr}
L^p(0,T;B_{\infty,1}^{-1/2})\cap C([0,T],B_{\infty,1}^{-1/2-\delta})\cap C^1([0,T],B^{-3/2-\delta}_{\infty,1})\quad \mathbf{P}\text{-a.s. for all } p\in[1,\infty).
\end{equation}
Moreover, for any $\varepsilon>0$ we can find an $\mathcal{F}$-measurable analytically weak solution $\theta$ such that
\begin{equation}\label{eq:33}
\|\theta\|_{C_TB^{-1/2-\delta}_{\infty,1}}\leq \|v\|_{C_TB^{-1/2-\delta}_{\infty,1}}+\varepsilon,
\end{equation}
 There are infinitely many such solutions $\theta$.
\et

The proof follows by   similar arguments as in the proof of Theorem \ref{thm:6.1} using the iterative
Proposition \ref{p:iteration2} formulated below.

As a consequence, we are able to prescribe an arbitrary initial as well as  terminal law.

\begin{corollary}\label{cor:law}
Let $T\geq 4$. Let $\Lambda_{0}$, $\Lambda_{T}$ be arbitrary Borel probability measures on $C^{\eta},$ $\eta>1/2$, both supported on functions with zero mean. There exist infinitely many $\mathcal{F}$-measurable solutions $\theta$ to \eqref{eq1} satisfying \eqref{eq:tr} and \eqref{eq:33} and such that the law of $\theta_{0}$ is given by $\Lambda_{0}$ and the law of $\theta({T})$ is given by $\Lambda_{T}$.
\end{corollary}

\begin{remark}
We note that by a straightforward modification of the convex integration construction,  the above results can be generalized to any $T>0$.
\end{remark}

Let us now show how the iteration leading to Theorem~\ref{thm:6.11} is set up. We use the same parameters as in Section \ref{s:it}.
 At each step $n$, a pair $(f_{\leq n},q_n)$ is constructed solving the following system
\begin{align}\label{induction ps1}
\begin{aligned}
-\partial_t \mathcal{R}f_{\leqslant n} + \Lambda f_{\leqslant n}
   \nabla^{\perp} f_{\leqslant n} +\mathcal{R} \Lambda^{- 1 + \alpha}
   P_{\leq \lambda_n}\xi &\approx \Lambda^{\gamma - 1} \nabla f_{\leqslant n} +
   \nabla q_n, \qquad t\in[0,T],\\
      f_{\leqslant n}(t)&=P_{\leq \lambda_n/3}\Lambda^{-1}v(T), \qquad t\in[T-2^{-n},T],\\
   f_{\leqslant n}(0)&=P_{\leq \lambda_n/3}\Lambda^{-1}v(0).\\
   \end{aligned}
   \end{align}
In other words, in addition to the initial condition and the noise, we also add the terminal value scale by scale. The overall structure of the iterative proposition is similar as above. The main difference is that at each iteration step, the perturbations are only added in the middle of the time interval $[0,T]$ and with increasing $n$  we are approaching both initial and terminal time using  suitably chosen cut-off functions.

Now we first assume that
\begin{align}\label{eq:u1}\|v\|_{C^1_TC^\eta}\leq N,\quad \|\xi\|_{C^{-1-\kappa}}\leq L.\end{align}
\begin{proposition}\label{p:iteration2}
Let $L,N\geq 1$ and assume \eqref{eq:u1}. There exists a choice of parameters $a, b, \beta$ such that the following holds true: Let $(f_{\leq n},q_{n})$ for some $n\in\N_{0}$ be an $\mathcal{F}$-measurable solution to \eqref{induction ps1} such that for any $t\in[0,T]$ the frequencies of $f_{\leq n}(t),q_n(t)$ are localized in a ball of radius $\leq 6\lambda_n$ and $\leq 12\lambda_n$, respectively, and for any $t\in[0,T]$
\begin{equation}\label{inductionv ps2}
\|f_{\leq n}(t)\|_{B^{1/2}_{\infty,1}}\leq13M_0M_L^{1/2}
\end{equation}
for a constant $M_0$ independent of $a,b,\beta,M_L$, and
\begin{align}\label{eq:C2}
\|f_{\leq n}\|_{C^1_TB^{1/2}_{\infty,1}}\leq M_0M^{1/2}_L\lambda_n,
\end{align}
for $\delta>\beta/2$
\begin{align}\label{induction wft}\|f_{\leq n}\|_{C^1B^{-1/2-\delta}_{\infty,1}}\leq M_L^{1/2}+M_0M^{1/2}_L\sum_{k=1}^n(\lambda_{k}^{-\delta}+\lambda_{k-1}^{-3/2-\delta-\eta})\leq M_L^{1/2}+6M_0M_L^{1/2},\end{align}
\begin{align}\label{eq:R2}
\|q_n(t)\|_{X}\leq
r_{n},\quad t\in [2^{-n+1},T-2^{-n+1}],
\end{align}
\begin{align}\label{bd:R1}
\|q_n(t)\|_{X}\leq \sum_{k=0}^nr_{k}\leq 3M_L,\quad t\in[0, 2^{-n+1})\cup (T-2^{-n+1},T].
\end{align}
 Then    there exists an $\mathcal{F}$-measurable solution $(f_{\leq n+1},q_{n+1})$ which solves \eqref{induction ps1} with $n$ replaced by $n+1$ and such that for $t\in[0,T]$ the frequencies of $f_{\leq n+1}(t),q_{n+1}(t)$ are localized in a ball of radius $\leq 6\lambda_{n+1}$ and $\leq 12\lambda_{n+1}$, respectively, and
 \begin{equation}\label{iteration ps2}
\|f_{\leq n+1}(t)-f_{\leq n}(t)\|_{B^{1/2}_{\infty,1}}\leq \begin{cases}
M_0 r_{n}^{1/2}+r^{1/2}_{n+1},& t\in (2^{-n+2},T-2^{-n+2}),\\
M_0 M_L^{1/2}+r^{1/2}_{n+1},& t\in (2^{-n-1},2^{-n+2}]\cup [T-2^{-n+2},T-2^{-n-1}),\\
r^{1/2}_{n+1},
&t\in [0,2^{-n-1}]\cup [T-2^{-n-1},T],
\end{cases}
\end{equation}
and for $\delta>\beta/2$
\begin{align}\label{induction w22}
\|f_{\leq n+1}-f_{\leq n}\|_{C_TB^{1/2-\delta}_{\infty,1}}\leq r^{1/2}_{n+1},
\end{align}
\begin{equation}\label{iteration R2}
\|q_{n+1}(t)\|_{X}\leq\begin{cases}
r_{n+1},& t\in {[2^{-n},T-2^{-n}]},\\
r_{n+1}+\|q_{n}\|_{C_TX},&t\in {[0,2^{-n})\cup (T-2^{-n},T]}.
\end{cases}
\end{equation}
Consequently, $(f_{\leq n+1},q_{n+1})$ obeys \eqref{inductionv ps2},  \eqref{eq:C2},  \eqref{induction wft}, \eqref{eq:R2} and \eqref{bd:R1} at the level $n+1$.
Moreover, for any $\eps>0$ we could choose the parameter $a$ large enough depending on $M_L, C_0$ such that
\begin{align}\label{induction w3}\|f_{\leq n+1}-f_{\leq n}\|_{C_TB^{1/2-\delta}_{\infty,1}}\leq \eps/2^{n+1}.
\end{align}
\end{proposition}

  Now, we start the iteration by letting $f_{\leq 0}=P_{\leq \lambda_0/3}\Lambda^{-1}v$. We obtain from \eqref{induction ps1}
  $$\nabla q_0 \thickapprox\Lambda f_{\leqslant 0}
   \nabla^{\perp} f_{\leqslant 0} +\mathcal{R} \Lambda^{- 1 + \alpha}
   P_{\leq \lambda_0}\xi -\Lambda^{\gamma - 1} \nabla f_{\leq0}-\partial_t \mathcal{R}f_{\leq0}  ,$$
   which by \eqref{eq:u1} implies (taking $q_{0}$ mean-free)
  $$\|q_0\|_{C_TX}\lesssim \|v\|_{C_TC^\eta}^2+\|v\|_{C_T^1C^\eta}+L\leq C(N^2+N+L)=M_L.$$
  Moreover,
  $$\|f_{\leq 0}\|_{C^1_TB^{1/2}_{\infty,1}}\leq N.$$
  Consequently, $(f_{\leq0},q_{0})$ obeys  \eqref{inductionv ps2}, \eqref{eq:C2}, \eqref{induction wft}, \eqref{eq:R2} and \eqref{bd:R1} at the level $0$.

Let us now prove  Proposition \ref{p:iteration2}. The parameters are chosen as in Section~\ref{s:p}.

\subsubsection{Mollification}

We intend to replace $f_{\leq n}$ by its time mollification. Let $\{\varphi_\varepsilon\}_{\varepsilon>0}$ be a family of  standard mollifiers with support of $\varphi$ in $(0,1)$. The one sided mollifier here is again used in order to preserve initial data. We extend $f_{\leq n}, q_n$ to $t<0$ by taking values at $t=0$. The equation also holds for $t<0$ as $\partial_t f_{\leq n}(0)=0$ by our construction. We define a mollification of $f_{\leq n}$, $q_n$  in  time by convolution as follows
	$$f_{\ell}=f_{\leq n}*_t\varphi_{\ell_{n+1}},\qquad
	q_\ell=q_n*_t\varphi_{\ell_{n+1}},$$
	where  $\ell:=\ell_{n+1}=\lambda_{n+1}^{-1}$.
We  mollify both sides of \eqref{induction ps1} and since $\ell<2^{-n-1}$ we obtain
\begin{align}\label{mo2}
\begin{aligned}
- \partial_t \mathcal{R}f_{\ell} + \Lambda f_{\ell} \nabla^{\perp}
   f_{\ell} +\mathcal{R} \Lambda^{- 1 + \alpha}P_{\leq\lambda_{n}} \xi
   &\approx  \Lambda^{\gamma - 1} \nabla f_{\ell} + \nabla q_{\ell} +
   R_{\tmop{com}} ,\qquad t\in [0,T],\\
   f_{\ell}(t)&=P_{\leq \lambda_n/3}\Lambda^{-1}v(T), \qquad t\in[T-2^{-n-1},T],\\
   f_\ell(0)&=P_{\leq\lambda_n/3}\Lambda^{-1}v(0).
   \end{aligned}
   \end{align}
   Here $R_{\tmop{com}}=\Lambda f_{\ell} \nabla^{\perp}
   f_{\ell}-(\Lambda f_{\leq n}\nabla^\bot f_{\leq n})*_t\varphi_{\ell}$.
  In view of \eqref{eq:C2} similar as \eqref{mol1} we obtain
 \begin{align}\label{mol2a}
  \begin{aligned}
\|f_{\ell}-f_{\leq n}\|_{C_TB_{\infty,1}^{1/2}}
  \leq \frac14r^{1/2}_{n+1},
  \end{aligned}
  \end{align}
  and
  \begin{align}\label{mol2b}
\|f_\ell\|_{C_T^1B_{\infty,1}^{1/2}}\leq \|f_{\leq n}\|_{C_T^1B_{\infty,1}^{1/2}}\leq M_0 M^{1/2}_L\lambda_n.
  \end{align}

\subsubsection{Construction of $f_{\leq n+1}$}

We first introduce a smooth cut-off function satisfying
\begin{align*}
\tilde{\chi}(t)=\begin{cases}
r_n,&t\in[ 2^{-n+2},T-2^{-n+2}],\\
\in (r_n,3M_L),& t\in ({3\cdot2^{-n}},2^{-n+2} )\cup (T-2^{-n+2},T-2^{-n+1}) ,\\
3M_L,&t\in [0,{3\cdot2^{-n}}]\cup [T-2^{-n+1},T].
\end{cases}
\end{align*}
Here in the middle interval we  smoothly interpolate such that
$\|\tilde{\chi}'\|_{C^{0}}\leq 3M_L2^{n}$ and similarly as above $\|q_\ell\|_X\leq\tilde{\chi}$.
Now, similarly as in Section \ref{sec:conf} we define
$$
\tilde{f}_{n+1}(t,x)=\sum_{j=1}^2a_{j,n+1}(t,x)\cos (5\lambda_{n+1}l_j\cdot x), \quad a_{j,n+1}=2\sqrt{\frac{\tilde{\chi}}{5\lambda_{n+1}}}P_{\leq \mu_{n+1}}\sqrt{C_0+\mathcal{R}_j^o\frac{q_{\ell}}{\tilde{\chi}}},
$$
where $l_1=(\frac35,\frac45)^T$, $l_2=(1,0)^T$.
As $L,N$  as well as the parameters $a,b,\beta$  are deterministic,   $\tilde{f}_{n+1}$ is $\mathcal{F}$-measurable.

We also introduce the following cut-off function
\begin{align*}
\chi(t)=\begin{cases}
0,& t\in [0,2^{-n-1}]\cup [T-2^{-n-1},T],\\
\in (0,1),& t\in (2^{-n-1},2^{-n} )\cup (T-2^{-n},T-2^{-n-1} ),\\
1,&t\in [2^{-n},T-2^{-n}].
\end{cases}
\end{align*}
Here in the middle interval we  interpolates smoothly such that it holds
$\|\chi'\|_{C^{0}}\leq 2^{n+1}$. Now, define $$f_{n+1}=\chi\tilde{f}_{n+1}+f_{n+1}^{in}:=\chi\tilde{f}_{n+1}+(P_{\leq \lambda_{n+1}/3}-P_{\leq \lambda_n/3})\Lambda^{-1}v,$$
which is mean zero and $\mathcal{F}$-measurable. We first have
\begin{align}\label{init}\|f_{n+1}^{in}\|_{C_TB_{\infty,1}^{1/2}}\leq\|f_{n+1}^{in}\|_{C^1_TB_{\infty,1}^{1/2}}\lesssim\lambda_n^{-1}\|v\|_{C^1_TC^\eta}\leq \frac12r^{1/2}_{n+1},
\end{align}
where we used $\beta b<2$.
In view of definition of $\tilde{f}_{n+1}$ we use similar calculation as in \eqref{fn} to obtain
\begin{align}\label{fnt}
\|\tilde{f}_{n+1}(t)\|_{B_{\infty,1}^{1/2}}
&\lesssim \lambda_{n+1}^{1/2}\left((C_0+1)\frac{\tilde\chi}{5\lambda_{n+1}}\right)^{1/2}
\nonumber\\&\leq
\begin{cases}
M_0r_n^{1/2},& t\in[2^{-n+2},T-2^{-n+2}],\\
 M_{0}M_{L}^{1/2},&t\in [0,2^{-n+2})\cup(T-2^{-n+2},T] .
 \end{cases}
\end{align}
Furthermore, similarly as in \eqref{fn3} for $t\in[0,T]$ we obtain
\begin{align}\label{fn3t}
\|\tilde{f}_{n+1}(t)\|_{B_{\infty,1}^{1/2-\delta}}
\leq \frac14r^{1/2}_{n+1}.
\end{align}
Here we used $\delta>\beta/2$ and we chose $a$ sufficiently large depending on $C_{0}$.

 The new solution $f_{\leq n+1}$ is defined as
$$f_{\leq n+1}:=f_\ell+f_{n+1},$$
which is also mean zero and $\mathcal{F}$-measurable.

\subsubsection{Inductive estimates for $f_{\leq n+1}$}
From construction we see that $\supp\widehat{f_{\leq n+1}}\subset \{|k|\leq 6\lambda_{n+1}\}$ and $f_{\leq n+1}(0)=P_{\leq \lambda_{n+1}/3}\Lambda^{-1}v(0)$ and $f_{\leqslant n+1}(t)=P_{\leq \lambda_{n+1}/3}\Lambda^{-1}v(T)$, $t\in[T-2^{-n-1},T]$. We first prove \eqref{iteration ps2}. Combining \eqref{init}, \eqref{mol2a} and \eqref{fnt} we obtain for $t\in(2^{-n+2},T-2^{-n+2})$
$$\|f_{\leq n+1}(t)-f_{\leq n}(t)\|_{B_{\infty,1}^{1/2}}\leq \|f_\ell-f_{\leq n}\|_{C_tB_{\infty,1}^{1/2}}+\|f_{n+1}(t)\|_{B_{\infty,1}^{1/2}}\leq M_0r_n^{1/2}+r^{1/2}_{n+1}.$$
Also combining \eqref{init} and \eqref{mol2a} and \eqref{fnt} we obtain for $t\in(2^{-n-1},2^{-n+2}]\cup [T-2^{-n+2},T-2^{-n-1})$
$$\|f_{\leq n+1}(t)-f_{\leq n}(t)\|_{B_{\infty,1}^{1/2}}\leq \|f_\ell-f_{\leq n}\|_{C_tB_{\infty,1}^{1/2}}+\|f_{n+1}(t)\|_{B_{\infty,1}^{1/2}}\leq M_0M_L^{1/2}+r^{1/2}_{n+1}.$$
Also combining \eqref{init} and \eqref{mol2a}  we obtain for $t\in[0, 2^{-n-1}]\cup [T-2^{-n-1},T]$
$$\|f_{\leq n+1}(t)-f_{\leq n}(t)\|_{B_{\infty,1}^{1/2}}\leq \|f_\ell-f_{\leq n}\|_{C_tB_{\infty,1}^{1/2}}+\|f_{n+1}(t)\|_{B_{\infty,1}^{1/2}}\leq r^{1/2}_{n+1}.$$
 Thus \eqref{iteration ps2} holds. Moreover, we have
\begin{align*}
\|f_{\leq n+1}(t)\|_{B_{\infty,1}^{1/2}}&\leq N+\sum_{k=0}^n\|f_{\leq k+1}(t)-f_{\leq k}(t)\|_{B_{\infty,1}^{1/2}}
\\&\leq N+\sum_{k=0}^n(M_0r_k^{1/2}+r^{1/2}_{k+1})+\sum_{k=0}^nM_0M_L^{1/2}1_{t\in (2^{-k-1},2^{-k+2}]\cup[T-2^{-n+2},T-2^{-n-1})}
\\&\leq N+\sum_{k=0}^n(M_0r_k^{1/2}+r^{1/2}_{k+1})+6M_0M_L^{1/2}.
\end{align*}
Thus \eqref{inductionv ps2} follows. \eqref{induction w22} and \eqref{induction w3} follow from \eqref{fn3t}, \eqref{init} and \eqref{mol2a}.  With \eqref{inductionv ps2} at hand, we shall estimate the time derivative. We use similar calculation as in \eqref{fn2} to obtain
\begin{align}\label{fn2t}
\|\partial_{t}\tilde{f}_{n+1}\|_{C_{T}B_{\infty,1}^{1/2}}&
\leq (M_0-1)M^{1/2}_L\lambda_{n+1}.\nonumber
\end{align}
Thus \eqref{eq:C2} at level of $n+1$ follows from \eqref{mol2b} and \eqref{init}. Moreover, \eqref{induction wft} follows from similar argument as \eqref{induction wf}.

\subsubsection{Inductive estimate for $q_{n+1}$}
The error term $q_{n+1}$ is defined the same way  as Section \ref{s:estq}.
First, we estimate each term for $t\in [2^{-n},T-2^{-n}]$. In this case the estimate of $q_{n+1}$ is the same as Section \ref{s:estq} and the only difference is that now have the following term:
$$
\begin{aligned}
&\|\Delta^{-1}\nabla \cdot \partial_t\mathcal{R} f^{in}_{n+1}\|_{C_TX}\lesssim  \frac{\log\lambda_{n+1}} {\lambda_{n}} \|v\|_{C_T^1C^\eta}<\frac1{12}r_{n+1}.
\end{aligned}$$

 For $t\in (2^{-n-1},2^{-n} )\cup (T-2^{-n},T-2^{-n-1})$ we have the extra term $\mathcal{R}\tilde{f}_{n+1}(t)\partial_{t}\chi$ and
$q_{n+1}^1=(1-\chi^2)q_\ell.$
Then we have
$$\|q_{n+1}^1(t)\|_X\leq\sup_{s\in[t-\ell,t]}\|q_n(s)\|_X\leq \|q_n\|_{C_TX},$$
and
$$\|\Delta^{-1}\nabla\cdot\mathcal{R}\tilde{f}_{n+1}(t)\partial_{t}\chi \|_X<\frac1{24}r_{n+1}.$$
 Moreover on $t\in[0,2^{-n-1}]\cup [T-2^{-n-1},T]$ we have %
 $$
\nabla q_{n+1}\approx\nabla q_\ell+R_{\rm{com}}{-\nabla\Lambda^{\gamma-1}f^{in}_{n+1}}+\nabla q_N+\Lambda f^{in}_{n+1}\nabla^\bot f_{\ell}+\Lambda f_{\ell}\nabla^\bot f^{in}_{ n+1}+\Lambda f_{n+1}^{in}\nabla^\perp f_{n+1}^{in}- \partial_t\mathcal{R} f^{in}_{n+1}.$$
 Except for the first term, the other ones could be bounded by the previous argument. Thus we deduce
   $$\|q_{n+1}^1(t)\|_X\leq r_{n+1}+\sup_{s\in[t-\ell,t]}\|q_n(s)\|_X.$$
   Hence \eqref{iteration R2}, \eqref{eq:R2} and  \eqref{bd:R1} hold.

\subsection{Coming down from infinity}\label{s:com}

The universality of the  bound \eqref{eq:32} in Theorem~\ref{thm:6.1} and \eqref{eq:33} in Theorem~\ref{thm:6.11} has an interesting consequence: a coming down from infinity with respect to the initial condition and the  perturbation of the system. In particular,   we can  construct solutions which completely forget the initial condition and are bounded independently of the perturbation $\zeta$.

\begin{corollary}\label{c:coming1}
Let $T\geq 4$,  $\varepsilon>0$, $\delta>0$ and let $\theta_{0}\in C^{\eta}$, $\eta>1/2$, $\mathbf{P}$-a.s. be $\mathcal{F}$-measurable and with zero mean $\mathbf{P}$-a.s. There exist infinitely many $\mathcal{F}$-measurable solutions $\theta$ to \eqref{eq1} such that
$
\theta(t)$, $t\geq T,$ is independent of $\theta_{0}$
and
\begin{equation}\label{eq:333}
\|\theta\|_{C_{b}([T,\infty),B_{\infty,1}^{-1/2-\delta})}\leq \varepsilon.
\end{equation}
\end{corollary}

The prescribed bound \eqref{eq:333} holds independently of the size of $(\theta_{0},\zeta)\in C^{\eta}\times C^{-2+\kappa}$. Accordingly, the result can be applied to possibly unbounded sequences $\zeta_{n}$, $\theta_{0,n}$, $n\in\mathbb{N}$, to obtain a sequence of solutions $\theta_{n}$ which is bounded in $C_{b}([T,\infty);B^{-1/2-\delta}_{\infty,1})$.
The result is sharp in the sense that convergence of solutions in $H^{-1/2}$   would already permit to pass to the limit in the nonlinear term in the equation,  leading to a contradiction if the data $\zeta_{n}$ or $\theta_{0,n}$ did not converge.

\section{Stationary solutions}
\label{s:in}

\subsection{Ergodic stationary solutions}
\label{s:erg}

In this section we are concerned with the long time behavior, existence of stationary solutions and their ergodic structure. Since non-uniqueness was shown in the previous section, the Markov semigroups given for a bounded measurable function $\varphi$ on $B^{-1/2}_{\infty,1}$ by
$$
P_{t}\varphi(\theta_{0})=\mathbf{E}[\varphi(\theta(t,\theta_{0}))]
$$
are not well defined. Therefore, we work in the framework of stationarity understood with respect to shifts on trajectories.
More precisely,
let $\cT:=C(\mR;B^{-1/2}_{p,1})\times C(\mR;B^{-2+\kappa}_{p,p})$ for some $2\leq p<\infty$ and $\kappa>0$ be the trajectory space and let $S_t$, $t\in\mR$, be the shift on trajectories, i.e.
$$
S_t(\theta,\zeta)(\cdot)=(\theta,\zeta)(\cdot+t),\quad t\in \mR,\quad (\theta,\zeta)\in\cT.
$$
We denote the set of probability measures on $\mathcal{T}$ by $\mathcal{P}(\mathcal{T})$. We formulate the notion of stationary solution. Here and in the sequel, $\zeta$ represents the noise part in \eqref{eq1}, i.e. $\zeta=\Lambda^{\alpha}\xi$ for a spatial white noise $\xi$ with $0<\alpha<1-3\kappa$.
{In order to  construct  stationary solutions we first note that by a straightforward modification of Theorem~\ref{thm:6.1} we can  also construct solutions to the initial value problem starting from $\theta_0$ at time $-S$ for any $S\geq0$. Furthermore, these solutions can be regarded  as trajectories in  $\cT$ by setting $(\theta,\zeta)(t)=(\theta,\zeta)(-S+4)$ for $t\leq -S+4$. Note that with the latter extension to trajectories on $\mR$ we lose the initial value at time $-S$, however, that is irrelevant for the construction of stationary solutions below. We made this choice in order to be able to apply the bound \eqref{est:fn} uniformly on $\mR$.}

\begin{definition}
We say that  $((\Omega,\mathcal{F},\mathbf{P}),\theta,\xi)$ is a stationary solution to \eqref{eq2} provided it satisfies  \eqref{eq2} in the analytically weak sense on $(-\infty,\infty)$, more precisely
\begin{align*}
&\langle \theta(t),\psi\rangle+\int_s^t\frac12\langle \Lambda^{-1/2}\theta,\Lambda^{1/2}[\mathcal{R}^\perp\cdot,\nabla \psi] \theta\rangle\dif r
		\\&\qquad= \langle \theta(s),\psi\rangle + \int_{s}^{t}\langle\zeta -\nu\Lambda^{\gamma } \theta
		,\psi\rangle\dif r,
	\quad\forall \psi \in C^{\infty}(\mathbb{T}^{2}),\  t\geq s.
\end{align*}
and its law is shift invariant, that is,
$$\mathcal{L}[S_{t}(\theta,\zeta)]=\mathcal{L}[\theta,\zeta]\qquad\text{ for all }\quad t\in\mR.$$
\end{definition}

As the next step, we show that the convex integration solutions from Section~\ref{s:1.1} generate stationary solutions. These solutions are obtained as limits of ergodic averages from a Krylov--Bogoliubov argument.

\bt\label{th:s1}
Let $\theta$ be a solution  starting from $\theta_0$ at time $-S$ for $S\geq4$ and  such that $\E[\|\theta_0\|^m_{C^\eta}]<\infty$ for any $m\geq1$. Then there exists a sequence $T_{n}\to\infty$ and
 a stationary  solution $((\tilde\Omega,\tilde{\mathcal{F}},\tilde{\mathbf{P}}),\tilde\theta,\tilde \zeta)$ to \eqref{eq1} such that for some $\tau\geq0$
 $$
 \frac{1}{T_{n}}\int_{0}^{T_{n}}\mathcal{L}[S_{t+\tau}(\theta,\zeta)] \dif t\to \mathcal{L}[\tilde\theta,\tilde \zeta]
 $$
 weakly in the sense of probability measures on $\mathcal{T}$ as $n\to\infty$. Moreover, it holds true for some $m\in\mathbb{N}$ {and some $\delta\in (0,1)$}
 \begin{equation}\label{eq:s}
\tilde{\mathbf{E}}\left[ \|\tilde\theta\|_{C_{b}({\mR};B^{-1/2+\delta}_{\infty,1})}+\|\tilde\theta\|_{C_{b}^{{1-\delta}}({\mR};B^{-3/2-\delta}_{\infty,1})}\right]\lesssim 1+ \mathbf{E}[\|\theta_{0}\|^{m}_{C^{\eta}}].
 \end{equation}
\et

\begin{proof}
Using  \eqref{est:fn} and the construction we know that there
	exists $\delta\in(0,1)$ and $m\in\mathbb{N}$ so that
\begin{align}\label{eq:B1}
\mathbf{E} \left[\| \theta \|_{C_{b}(\mR;B_{\infty, 1}^{- 1 / 2 +
				\delta})}+\| \theta \|_{C_{b}^{1-\delta}(\mR; B^{- 3 / 2 - \delta}_{\infty, 1})}\right]\lesssim 1+ \mathbf{E}[\|\theta_{0}\|^{m}_{C^{\eta}}], \end{align}
For $T\geq0$ and $\tau\geq0$ we define the ergodic average as the probability measure on $\cT$
	\begin{align*}
		\nu_{\tau,T}:=\frac1T\int_0^T\cL[S_{t+\tau}(\theta,\zeta)]\dif t,
	\end{align*}
	and we show that the family $\nu_{\tau,T}$, $T\geq 0$, is tight. To this end,
we define for $R>0$ and $\kappa>0$
	$$B_R:=\Big\{g=(g_1,g_2)\in\cT;\,\|g_1\|_{C_{b}^{1-\delta}(\mR; B^{- 3 / 2 - \delta}_{\infty, 1})}+\|g_1\|_{C_{b}(\mR;
		B_{\infty, 1}^{- 1 / 2 + \delta})}+\|g_2\|_{C^\kappa_{b}(\mR; B^{-2+2\kappa}_{p,p})}\leq R\Big\},$$
	which is relatively compact in $\cT$.
	Then we use \eqref{eq:B1} and $\E[\|\zeta\|_{ B^{-2+2\kappa}_{p,p}}]\lesssim \E[\|\xi\|_{ B^{-1-\kappa}_{p,p}}]$ to have
\begin{align*}
		\nu_{\tau,T} (B_R^c) &= \left( \frac{1}{T} \int_0^T \cL[S_{t+\tau} (\theta,\zeta)]  \dif t\right) (B^c_R)\\
		&\leqslant \frac{1}{R T} \int_0^T \mathbf{E} \left[\| \theta(t+\tau+\cdot) \|_{C^{1-\delta}_{b}(\mR; B^{- 3 / 2 - \delta}_{\infty, 1})} + \|  \theta(t+\tau+\cdot) \|_{C_{b}(\mR;
			B_{\infty, 1}^{- 1 / 2 + \delta})}+\|\zeta\|_{ B^{-2+2\kappa}_{p,p}}\right] \dif t \\&\lesssim \frac{C(\theta_0)}{R} .
\end{align*}

So for $\tau\geq0$ the measures $\nu_{\tau,T-\tau}$, $T \geqslant 0$, are tight on $\cT$ and hence there is a weakly
	converging subsequence, i.e. there is a subsequence $T_n\to\infty$ and $\nu\in \cP(\cT)$ such that $\nu_{\tau,T_n-\tau}\to \nu$ weakly in $\cP(\cT)$. By a similar argument as in \cite[Lemma 5.2]{BFH20e} we know that for all $\tau\geq0$,  $\nu_{\tau,T_n-\tau}\to \nu$ weakly in $\cP(\cT)$, i.e. $\nu$ is independent of the choice of $\tau$. Furthermore, we take a sequence $\tau_m\to\infty$ and consider $\nu_{\tau_m,T_n-\tau_m}$, $m,n\in\mN$. Denote by $d$ the metric on $\cP(\cT)$ metrizing the weak convergence. For $m\in\mN$ we could find $n(m)$ such that
	$d(\nu_{\tau_m,T_{n(m)}-\tau_m},\nu)<\frac1m$. Hence, it follows that
	$\nu_{\tau_m,T_{n(m)}-\tau_m}\to \nu$ weakly in $\cP(\cT)$, as $m\to\infty$.
	
	 Define for every $T\geq0$ the set
\begin{align*}
	A_T=\Big\{(\theta,\zeta)\in \cT;\
&\langle \theta(t),\psi\rangle+\int_s^t\frac12\langle \Lambda^{-1/2}\theta,\Lambda^{1/2}[\mathcal{R}^\perp\cdot,\nabla \psi] \theta\rangle\dif r
		\\&= \langle \theta(s),\psi\rangle + \int_{s}^{t}\langle\zeta -\nu\Lambda^{\gamma } \theta
		,\psi\rangle\dif r,
	\quad\forall \psi \in C^{\infty}(\mathbb{T}^{2}),\  t\geq s\geq -T\Big\}.
\end{align*}
Since $(\theta,\zeta)$ in the statement of the theorem satisfies the equation, we
have for all $t\geq 0$ and $\tau\geq T+4$
\begin{align*}
	\cL[S_{t+\tau}(\theta,\zeta)](A_T)=1.
\end{align*}
Hence, for $m$ large enough, $\nu_{\tau_m,T_{n(m)}-\tau_m}(A_T)=1$. By Skorokhod theorem, we  find a probability space $(\tilde \Omega,\tilde{\mathcal{F}},\tilde{\mathbf{P}})$ and on it a sequence of random variables $(\tilde\theta^m,\tilde\zeta^m)$, $m\in\mathbb{N}$, such that $\cL[\tilde\theta^m,\tilde\zeta^m] ={\nu_{\tau_{m},T_{n(m)}-\tau_{m}}}$ %
and  $(\tilde\theta^m,\tilde\zeta^m)$ satisfy equation \eqref{eq1} on $[-T,\infty)$. Moreover, there is a  random variable $(\tilde\theta,\tilde\zeta)$ having the law $\cL[\tilde\theta,\tilde\zeta]=\nu$ so that
$$
(\tilde\theta^m,\tilde\zeta^m)\to(\tilde\theta,\tilde\zeta)\qquad \tilde{\mathbf{P}}\text{-a.s. in }\mathcal{T}.
$$
Thus, we can pass to the limit in the equation to deduce that $\nu$ is a law of a solution on $[-T,\infty)$ for every $T\geq0$.

	As a consequence of the same argument as in \cite[Lemma 5.2]{BFH20e} we deduce that the limit probability measure $\nu$ is shift invariant in the desired sense, i.e.  for all $G\in C_b(\cT)$ and all $r\in\mR$
\begin{align*}
			\int_{\cT} G\circ S_r(\theta,\zeta) \dif \nu(\theta,\zeta)=\int_{\cT} G(\theta,\zeta) \dif \nu(\theta,\zeta).
	\end{align*}
	It remains to observe that \eqref{eq:s} follows from a lower-semicontinuity argument, equality of laws and \eqref{eq:B1}
	\begin{align}\label{eq:ee1}
	\begin{aligned}
&	\tilde{\mathbf{E}}\left[ \|\tilde\theta\|_{C_{b}(\mR;B^{-1/2+\delta}_{\infty,1})}+\|\tilde\theta\|_{C_{b}^{1-\delta}(\mR;B^{-3/2-\delta}_{\infty,1})}\right]
\\&\qquad\leq \liminf_{n\to\infty}\tilde{\mathbf{E}}\left[ \|\tilde\theta^{n}\|_{C_{b}(\mR;B^{-1/2+\delta}_{\infty,1})}+\|\tilde\theta^{n}\|_{C_{b}^{1-\delta}(\mR;B^{-3/2-\delta}_{\infty,1})}\right]
\\&\qquad= \liminf_{n\to\infty}\frac{1}{T_{n}}\int_{0}^{T_{n}}{\mathbf{E}}\left[ \|S_{s+\tau}\theta\|_{C_{b}(\mR;B^{-1/2+\delta}_{\infty,1})}+\|S_{s+\tau}\theta\|_{C^{1-\delta}_{b}(\mR;B^{-3/2-\delta}_{\infty,1})}\right]\dif s\leq C(\theta_{0}).
	\end{aligned}
	\end{align}
\end{proof}

\begin{corollary}\label{cor:4.3}
There are infinitely many non-Gaussian stationary solutions.
\end{corollary}

\begin{proof}
By \eqref{induction w1} we know that the solution $\theta$ obtained in Theorem \ref{thm:6.1} extended to $\mR$ as explained above satisfies
$$\|\theta-P_{\leq \lambda_0/3}\theta_0\|_{C_{b}(\mR;B^{-1/2-\delta}_{\infty,1})}\leq \eps.$$
Then for a deterministic $\theta_0$ with only finitely many non-zero Fourier modes we could choose $\lambda_0$ large enough such that $P_{\leq \lambda_0/3}\theta_0=\theta_0$ and consequently  $\nu_{\tau,T-\tau}(\theta:\|\theta-\theta_0\|_{C_{b}(\mR;B^{-1/2-\delta}_{\infty,1})}\leq \eps)=1$. This  implies that  $\tilde{\theta}^n$ in the proof of Theorem \ref{th:s1} satisfies {$\tilde{\mathbf{P}}$-a.s.}
$$\|\tilde{\theta}^n-\theta_0\|_{C_{b}(\mR;B^{-1/2-\delta}_{\infty,1})}\leq \eps.$$
Taking the limit we obtain by lower-semicontinuity {$\tilde{\mathbf{P}}$-a.s.}
$$\|\tilde{\theta}-\theta_0\|_{C_{b}(\mR;B^{-1/2-\delta}_{\infty,1})}\leq \eps.$$
Hence, we can choose different deterministic $\theta_0$ such that the corresponding stationary solution lives in different balls with respect to the ${C_{b}(\mR;B^{-1/2-\delta}_{\infty,1})}$-norm. Hence there exist infinitely many stationary solutions and the solution is not Gaussian by the same arguments as in the proof of Theorem \ref{thm:6.1}.
\end{proof}

\begin{definition}
A stationary solution $((\Omega,\mathcal{F},\mathbf{P}),\theta,\zeta)$ is ergodic provided
\[ \mathcal{L}[\theta,\zeta] (B) = 1 \quad \text{or} \quad \mathcal{L}[\theta,\zeta] (B) = 0 \quad \text{for all } B
\subset\mathcal{T} \text{ Borel and shift invariant}. \]
\end{definition}

\begin{remark}
Note that our notion of ergodicity of a solution $((\Omega,\mathcal{F},\mathbf{P}),\theta,\zeta)$ coincides with the ergodicity of the dynamical system
\[ \big(\mathcal{T}, \mathcal{B} (\mathcal{T}), (S_t, t \in \mR), \mathcal{L} [{\theta},\zeta]\big), \]see e.g. \cite[Chapter 1]{DPZ96}. In the Markovian framework, ergodicity of $({\theta},\zeta)$ in particular implies ergodicity of the  projection $\mu=\mathcal{L} [\theta]\circ\pi^{-1}_{0}$, with $\pi_{0}:C(\mR;B^{-1/2}_{p,1})\to B^{-1/2}_{p,1}$ is the projection onto $t=0$, with respect to the corresponding Markov semigroup, cf.  \cite[Chapter 3]{DPZ96}.
\end{remark}

By a general result applied also in \cite{BFH20e,FFH21} and using Theorem \ref{th:s1} we obtain the following.

\bt\label{thm:4.5}
There exist $K>0$ and an ergodic stationary solution $((\Omega,\mathcal{F},\mathbf{P}),\theta,\zeta)$ satisfying for some ${\delta}{\in (0,1)}$
\begin{equation}\label{eq:s55}
{\mathbf{E}}\left[ \|\theta\|_{C_{b}(\mR;B^{-1/2+\delta}_{\infty,1})}+\|\theta\|_{C_{b}^{1-\delta}(\mR;B^{-3/2-\delta}_{\infty,1})}\right]\leq K,
 \end{equation}
 and for a given deterministic $\theta_0$ with only finitely many non-zero  Fourier modes it holds $\mathbf{P}$-a.s.
\begin{equation}\label{eq:s56}\|\theta-\theta_0\|_{C_{b}(\mR;B^{-1/2-\delta}_{\infty,1})}\leq 1.\end{equation}
In particular, there are infinitely many non-Gaussian ergodic solutions.
\et

\begin{proof}
In view of Theorem~\ref{th:s1}, this is a consequence of the classical Krein--Milman argument. More precisely, it is enough to note that the set of all laws of  stationary solutions satisfying \eqref{eq:s55} and \eqref{eq:s56} is non-empty, convex, tight and closed which follows from the same argument as the proof of Theorem \ref{th:s1} and Corollary \ref{cor:4.3}. Hence there exist an extremal point. By a classical contradiction argument,  one shows that it is the law of an ergodic stationary solution. Also \eqref{eq:s56} implies that the law is not Gaussian by the same argument as the proof of Theorem \ref{thm:6.1}. {Non-uniqueness is achieved by choosing different $\theta_{0}$.}
\end{proof}

In Section~\ref{s:stat}, namely in Theorem~\ref{thm:4.1}, we construct infinitely many steady state solutions. These are particular examples of stationary solutions whose trajectories  are constant in time. At this point, we are not able to say whether the stationary solutions from Theorem~\ref{th:s1} are time dependent. However, we can deduce that time dependent stationary solutions exist and are non-unique.

\begin{corollary}\label{cor:4.6}
It holds
\begin{enumerate}
\item 	Steady state solutions are not ergodic.
\item There exist  time
	dependent stationary solutions.
\end{enumerate}
\end{corollary}

\begin{proof}
	Let $\theta \in C(\mR;B^{- 1 / 2}_{\infty, 1})$ $\tmmathbf{P}$-a.s. be an arbitrary
	steady state solution, not necessarily constructed by Theorem \ref{thm:4.1}.
	For every Borel set $B\subset B^{-1/2}_{\infty,1}$, $C(\mR;B)\times C(\mR;B^{-2+\kappa}_{p,p}) \subset \cT$ is a shift invariant
	Borel set in $\mathcal{T}$. Hence, it is enough to find such a $B \subset
	B^{- 1 / 2}_{\infty, 1}$ so that $\mathbf{P} \circ \theta^{- 1} (C(\mR;B))\in (0,
		1)$. Since $\theta$ satisfies the equation in the analytically weak sense,
	it follows that $\theta$ is not a.s. constant. Hence there is $k \in
	\mathbb{Z}^2$ so that $\hat{\theta} (k)$ is not a.s. a constant and
	consequently there exists $K \in \mathbb{R}$ so that
	\[ \mathbf{P} (\hat{\theta} (k) > K) \in (0, 1) . \]
	Setting $B = \{ g \in B^{- 1 / 2}_{\infty, 1} ; \hat{g} (k) > K \}$
	implies the first claim.
	
	Since there exists at least one ergodic stationary solution, it  must therefore be time dependent. Hence the second claim follows.
\end{proof}

So far we have seen that the long time behavior of solutions sensitively depends on the initial condition and there exist multiple ergodic stationary solutions. The natural question is then whether the so-called ergodic hypothesis is valid for a solution $((\Omega,\mathcal{F},\mathbf{P}),\theta,\zeta)$, namely, whether the limit of the ergodic averages
\begin{equation}\label{eq:0}
\lim_{T\to\infty}\frac1T\int_{0}^{T}F(\theta(t))\,\dif t
\end{equation}
exists for any bounded continuous functional $F$ on $B^{-1/2}_{p,1}$ and is given by an ensemble average with respect to a certain probability measure.

If $((\Omega,\mathcal{F},\mathbf{P}),\theta,\zeta)$ is an ergodic  stationary solution then the desired a.s. convergence as well as convergence in $L^{1}(\Omega)$  is a consequence of Birkhoff--Khinchin's ergodic theorem (see \cite[Chapter 39]
{Kol91}) and the limit \eqref{eq:0} is  identified with the ensemble average
$$
\mathbf{E}[F(\theta(0))]=\int_{B^{-1/2}_{p,1}}F\dif \mu,
$$
where $\mu=\cL[\theta]\circ \pi_{0}^{-1}$ with $\pi_{0}:C(\mR;B^{-1/2}_{p,1})\to B^{-1/2}_{p,1}$ is the projection onto $t=0$.

Nevertheless, for a general (non-stationary) solution it is not even  clear whether the limit \eqref{eq:0} exists. Recall that for our convex integration solution we could obtain existence of a converging subsequence for  probability measures associated to the ergodic averages. The ergodic hypothesis then boils down to the question whether all subsequences converge to the same limit, in other words, whether there is a unique stationary solution generated by the ergodic averages.

\subsection{Steady state solutions}
\label{s:stat}

In this section we are concerned with steady state solutions, i.e. solutions independent of time.
The construction in this case is much simpler as there is no time mollification in each iteration step. Additionally, there are no error terms coming from the initial and terminal condition and the time derivative and we can start iteration simply from $f_{\leq0}=0$. During each step we do not need the cut-off functions $\tilde\chi$, $\chi$. We only formulate the main iteration result and the main theorem, the proofs follow the lines of Section~\ref{s:1.1}.

\begin{proposition}\label{p:iteration1}
Let $L\geq 1$ and assume \eqref{eq:u0}. There exists a choice of parameters $a, b, \beta$ such that the following holds true: Let $(f_{\leq n},q_{n})$ for some $n\in\N_{0}$ be an $\mathcal{F}$-measurable stationary solution to \eqref{induction ps} such that the frequencies of $f_{\leq n},q_n$ are localized to $\leq 6\lambda_n$ and $\leq 12\lambda_n$, respectively, and
\begin{equation}\label{inductionv ps1}
\|f_{\leq n}\|_{B_{\infty,1}^{1/2}}\leq 3M_0M_L^{1/2}
\end{equation}
for a  constant $M_0$ independent of $a,b,\beta, M_{L}$, and
\begin{align}\label{eq:R1}
\|q_n\|_{X}\leq
r_{n}.
\end{align}
 Then    there exists an $\mathcal{F}$-measurable stationary $(f_{\leq n+1},q_{n+1})$ which solves \eqref{induction ps} with $n$ replaced by $n+1$ and such that the frequencies of $f_{\leq n+1},q_{n+1}$ are localized to $\leq 6\lambda_{n+1}$ and $\leq 12\lambda_{n+1}$, respectively, and
 \begin{equation}\label{iteration ps1}
\|f_{\leq n+1}-f_{\leq n}\|_{B_{\infty,1}^{1/2}}\leq  M_0r_{n}^{1/2},
\end{equation}
\begin{equation}\label{iteration R1}
\|q_{n+1}(t)\|_{X}\leq
r_{n+1}.
\end{equation}
Consequently, $(f_{\leq n+1},q_{n+1})$ obeys \eqref{inductionv ps1},  at the level $n+1$.
Moreover, for any $\eps>0,\delta>\beta/2$ we could choose $a$ large enough depending on $M_L, C_0$ such that
\begin{align}\label{induction w2}
\|f_{\leq n+1}-f_{\leq n}\|_{B_{\infty,1}^{1/2-\delta}}\leq \eps/2^{n+1}.
\end{align}
\end{proposition}

 The iteration is initiated  by letting $f_{\leq 0}=0$. Then we obtain from \eqref{induction ps}
  $$\nabla q_0 =\mathcal{R} \Lambda^{- 1 + \alpha}
   P_{\leq \lambda_0}\xi ,$$
   which implies that
  $$\|q_0\|_X\lesssim L=M_L.$$
Consequently, $(f_{\leq0},q_{0})$ obeys \eqref{inductionv ps1} and  \eqref{eq:R1} at the level $0$.

By a repeated application of Proposition~\ref{p:iteration1} we deduce the following result.

  \bt\label{thm:4.1}
There exists an $\mathcal{F}$-measurable steady state analytically weak solution $\theta$ to \eqref{eq1} which belongs to $B_{\infty,1}^{-1/2}$  $\mathbf{P}$-a.s. Moreover, for any $\varepsilon>0$ we can find an $\mathcal{F}$-measurable steady state analytically weak solution $\theta$ such that $\|\theta\|_{B_{\infty,1}^{-1/2-\delta}}\leq \varepsilon.$ This also implies that the law of the solution is not Gaussian. There are infinitely many such solutions $\theta$. In particular, this gives infinitely many steady state solutions to the corresponding elliptic  and wave equation.

\et
\begin{proof} Most of the proof  is similar as the above. Let us only discuss non-uniqueness. The formula for $f$ reads as
$$f = \sum_{n = 0}^{\infty} \sum_{j = 1}^2 2 \sqrt{\frac{r_n}{5 \lambda_{n +
   1}}} \left( P_{\leqslant \mu_{n + 1}} \sqrt{C_0 +\mathcal{R}^o_j
   \frac{q_n}{r_n}} \right) \cos (5 \lambda_{n + 1} l_j \cdummy x) . $$
This readily implies that the solution is not unique by choosing different $C_0$, because  this is an almost explicit Fourier series.
\end{proof}

\section{Extensions to other singular models}\label{other}

We revisit some of the other systems where convex integration yields existence of non-unique solutions. It turns out that the results also apply to certain irregular spatial perturbations. In particular, we focus on fractional Navier--Stokes system in $d=3$ and the power law Navier--Stokes system in $d\geq3$. In these cases, we are able to treat spatial perturbations $\zeta$ of regularity $C^{-1+\kappa}$ for any $\kappa>0$. The stronger requirement  on regularity of $\zeta$ compared to our results on the SQG equation comes from the fact, that the error term now only gains one derivative whereas for the SQG equation it gained two.

\subsection{Fractional Navier--Stokes equation}\label{s:5.1}
Let us consider the fractional Navier--Stokes equations on $\mathbb{T}^3$. The  equations govern the time evolution of the fluid velocity $u$ and  read as
\begin{equation}
\label{1}
\aligned
 \partial_{t} u+\div(u\otimes u)+\nabla P&=-(-\Delta)^\gamma u +\zeta,
\\
\div u&=0,\\
u(0)&=u_{0}.
\endaligned
\end{equation}
Here $P$ is the associated pressure,  $\zeta\in C^{-1+\kappa}(\mathbb{T}^3)$ with mean zero and  $\kappa\in(0,1), \gamma\in(0,1].$ Hence in particular, we may consider $\zeta=\Lambda^{\alpha}\xi$ with $\alpha<-1/2$ and $\xi$ being the space white noise on $\mathbb{T}^{3}$. The same approach actually also applies to the Euler setting. When $\gamma>\frac{2-\kappa}{4}$ the equation is subcritical, when $\gamma=\frac{2-\kappa}{4}$
it is critical and when $\gamma<\frac{2-\kappa}{4}$ it is supercritical. This case we can also use convex integration from \cite{BV19a, BV19} to obtain existence and nonuniqueness of global solutions from every $u_0\in L_\sigma^2$ (see also \cite{HZZ21a}). Here and in the following, $L^2_\sigma$ means $L^2$ space with divergence free condition. Although the result in \cite{BV19a, BV19, HZZ21a} is for $\gamma=1$, it also holds for $\gamma<1$ since fractional Laplacian is easier to  control and the rest of the proof does not change. Let $P_t$ be the semigroup associated with $(-\Delta)^\gamma$ and $z=P_tu_0$. In particular, we consider  the following Navier--Stokes--Reynolds system
\begin{equation*}
\aligned
\partial_{t}v^n+\div((v^n+z^n)\otimes (v^n+z^n))+\nabla P^n&=-(-\Delta)^\gamma v^n +\zeta+\div R^n,
\\
\div v^n&=0,\\
v^n(0)&=0.
\endaligned
\end{equation*}
Here $z^n$ is a suitable projection of $z$ (see \cite{HZZ21a} for more details).
Note that unlike in the setting of the SQG equation, here we do not need to add the noise  scale by scale but we can include the full $\zeta$ in the iterative equation above.
Starting with $v^0=0$ we obtain
$$\div R^0-\nabla p^0=\div (z^0\otimes z^0)-\zeta$$
which implies that
  $$\|R^0\|_{CL^1}\lesssim \|u_0\|_{L^2}^2+\|\zeta\|_{C^{-1+\kappa}}.$$
  Then during each iteration we need to control $\div R=\psi_\ell*\zeta-\zeta$  with $\psi_\ell$ being a mollifier in space. As $\psi_\ell*\zeta-\zeta$ is of mean zero, we have
$$\|R\|_{CL^1}\lesssim \|\psi_\ell*\zeta-\zeta\|_{C^{-1+\kappa/2}}\lesssim \|\zeta\|_{C^{-1+\kappa}} \ell^{\kappa/2},$$
which is small as $\ell$ is small. We can choose $\beta$ in \cite{BV19a, BV19} small enough such that the new error can be controlled during each iteration. By similar arguments as above and changing to the convex integration  construction in \cite{BV19a, BV19, HZZ21a}, we have the following result.

 \bt\label{thm:5.1}
 Fix $\gamma\in (0,1], \kappa\in (0,1)$.
For any   initial condition $u_0\in L^2_\sigma $ $\mathbf{P}$-a.s.
there exists an $\mathcal{F}$-measurable analytically weak solution $u$ to \eqref{1} with $u(0)=u_0$ which belongs to $L^p_{\rm{loc}}(0,\infty;L^2)\cap C([0,\infty);L^{31/30})$ $\mathbf{P}$-a.s. for all $p\in[1,\infty)$. There are infinitely many such solutions $u$.

\et

\subsection{Power law fluids}
Let us consider the power law Navier--Stokes equations on $\mathbb{T}^d$ with $d\geq3$. The  equations   read as
\begin{equation}
\label{2}
\aligned
\partial_{t} u+\div(u\otimes u)+\nabla P&=\div \mathcal{A}(D u) +\zeta,
\\
\div u&=0,\\
u(0)&=u_{0}.
\endaligned
\end{equation}
Here $P$ is the associated pressure, $Du=\frac12(\nabla u+\nabla^Tu)$,  $\zeta\in C^{-1+\kappa}(\mathbb{T}^d)$ with mean zero and  $\kappa\in(0,1).$ Here, we may therefore consider $\zeta=\Lambda^{\alpha}\xi$ for $\alpha<-d/2+1$ if $\xi$ is a space white noise in dimension $d$. $\mathcal{A}$ is given by the following power law $$\mathcal{A}(Q)=(\nu_0+\nu_1|Q|)^{q-2}Q,$$
for some $\nu_0$, $\nu_1\geq0,q\in (1,\infty)$.
This case is degenerate  if $\nu_0=0$ and we cannot use regularity structures or paracontrolled calculus. However, it is possible to  use convex integration from \cite{BMS20} to obtain existence and nonuniqueness of global solutions from every $u_0\in L_\sigma^2\cap W^{1,\max(q-1,1)}$ when $q<\frac{3d+2}{d+2}$ (see also \cite{LZ22}).  More precisely, let $z=e^{t\Delta}u_0$ and  consider  the following Reynolds system 
\begin{equation*}
\aligned
\partial_{t}v^n+\div((v^n+z^n)\otimes (v^n+z^n))+\nabla P^n&=\div\mathcal{A}(Dv^n+Dz^n)-\Delta z^n +\zeta+\div R^n,
\\
\div v^n&=0,\\
v^n(0)&=0.
\endaligned
\end{equation*}
Here $z^n$ is again a suitable projection of $z$ (see \cite{LZ22} for more details).
Starting with $v^0=0$ we obtain
$$\div R^0-\nabla p^0=\div (z^0\otimes z^0)-\div(\mathcal{A}(D z^0))+\Delta z^0-\zeta,$$
which implies that
  $$\|R^0\|_{CL^1}\lesssim \|u_0\|_{L^2}^2+\|u_0\|^{q-1}_{W^{1,\max(q-1,1)}}+\|u_0\|_{W^{1,\max(q-1,1)}}+\|\zeta\|_{C^{-1+\kappa}}.$$
  Then during each iteration we need to control $\div R=\psi_\ell*\zeta-\zeta$ with $\psi_\ell$ being  a mollifier in space which is similar to Section \ref{s:5.1}.  Hence using \cite{BMS20, LZ22} we have the following result.
 \bt\label{thm:3.1}Let $q<\frac{3d+2}{d+2}$.
For any   initial condition $u_0\in L^2_\sigma \cap W^{1,\max(q-1,1)}$ $\mathbf{P}$-a.s.
there exists an $\mathcal{F}$-measurable analytically weak solution $u$ to \eqref{2} with $u(0)=u_0$ which belongs to $L^p_{\rm{loc}}(0,\infty;L^2)\cap C([0,\infty); W^{1,\max(q-1,1)})$ $\mathbf{P}$-a.s. for all $p\in[1,\infty)$. There are infinitely many such solutions $u$.

\et

%

%

\def\cprime{$'$} \def\ocirc#1{\ifmmode\setbox0=\hbox{$#1$}\dimen0=\ht0
  \advance\dimen0 by1pt\rlap{\hbox to\wd0{\hss\raise\dimen0
  \hbox{\hskip.2em$\scriptscriptstyle\circ$}\hss}}#1\else {\accent"17 #1}\fi}

\end{document}